\documentclass[11pt]{article}
\pdfoutput=0
\usepackage{fullpage}
\usepackage{amsmath,amsthm,amssymb}
\usepackage[ruled,lined]{algorithm2e}
\usepackage{graphicx,psfrag}
\usepackage[captionskip=1ex]{subfig}

\newcommand{\reals}{\mathbb{R}}
\newcommand{\minimize}{\mathop{\mathrm{minimize}{}}}

\newcommand{\argmin}{\mathop{\mathrm{arg\,min}{}}}

\newcommand{\supp}{\mathrm{supp}}
\newcommand{\sgn}{\mathrm{sgn}}

\newcommand{\gammaInc}{\gamma_\mathrm{inc}}
\newcommand{\gammaDec}{\gamma_\mathrm{dec}}
\newcommand{\Lmin}{L_\mathrm{min}}
\newcommand{\SoftThresholding}{\mathrm{soft}}

\newtheorem{lemma}{Lemma}
\newtheorem{theorem}{Theorem}
\newtheorem{corollary}{Corollary}
\newtheorem{definition}{Definition}
\newtheorem{assumption}{Assumption}

\usepackage{color}

\title{A Proximal-Gradient Homotopy Method for the
Sparse Least-Squares Problem}

\author{
Lin Xiao\thanks{
Machine Learning Department, Microsoft Research, Redmond, WA 98052. 
Email: \texttt{lin.xiao@microsoft.com}} 
\and
Tong Zhang\thanks{
Department of Statistics, Rutgers University, Piscataway, NJ, 08854. 
Email: \texttt{tzhang@stat.rutgers.edu}}
}

\begin{document}
\maketitle

\begin{abstract} 
We consider solving the $\ell_1$-regularized least-squares ($\ell_1$-LS) 
problem in the context of sparse recovery, for applications such as 
compressed sensing.
The standard proximal gradient method, also known as iterative 
soft-thresholding when applied to this problem, has low computational cost 
per iteration but a rather slow convergence rate.
Nevertheless, when the solution is sparse, it often exhibits fast linear 
convergence in the final stage. 
We exploit the local linear convergence using a homotopy continuation strategy,
i.e., we solve the $\ell_1$-LS problem for a sequence of decreasing values 
of the regularization parameter, and use an approximate solution at the end 
of each stage to warm start the next stage.
Although similar strategies have been studied in the literature, 
there have been no theoretical analysis of their global iteration complexity.
This paper shows that under suitable assumptions for sparse recovery, 
the proposed homotopy strategy ensures that 
all iterates along the homotopy solution path are sparse.
Therefore the objective function is effectively strongly convex along the 
solution path, and geometric convergence at each stage can be established.
As a result, the overall iteration complexity of our method is 
$O(\log(1/\epsilon))$ for finding an $\epsilon$-optimal solution, which
can be interpreted as global geometric rate of convergence.
We also present empirical results to support our theoretical analysis.
\end{abstract}

\section{Introduction}
In this paper, we propose and analyze an efficient numerical method for 
solving the \emph{$\ell_1$-regularized least-squares} ($\ell_1$-LS) problem 
\begin{equation}\label{eqn:l1-LS}
\minimize_x \quad \frac{1}{2}\|A x - b\|_2^2 + \lambda \|x\|_1,
\end{equation}
where $x\in\reals^n$ is the vector of unknowns, 
$A\in\reals^{m\times n}$ and $b\in\reals^m$ are the problem data,
and $\lambda>0$ is a regularization parameter.
Here $\|\cdot\|_2$ denotes the standard Euclidean norm,
and $\|x\|_1=\sum_i|x_i|$ is the $\ell_1$ norm of~$x$.
This is a convex optimization problem, and we use 
$x^\star(\lambda)$ to denote its (global) optimal solution.
Since the $\ell_1$ term promotes sparse solutions, we also refer
problem~(\ref{eqn:l1-LS}) as the \emph{sparse least-squares} problem.

The $\ell_1$-LS problem has important applications in machine learning, 
signal processing, and statistics; see, e.g., 
\cite{Tibshirani96, ChenDS98, BrucksteinDonohoElad09}.
It received revived interests in recent years due to the emergence of
\emph{compressed sensing} theory, which builds upon the fundamental idea that
a finite-dimensional signal having a sparse or compressible representation 
can be recovered from a small set of linear, nonadaptive measurements
\cite{CandesRT06,CandesTao06,Donoho06}.
We are especially interested in solving the $\ell_1$-LS problem in such a 
context, with the goal of recovering a sparse vector under measurement noise.
More precisely, we assume~$A$ and~$b$ in~(\ref{eqn:l1-LS}) are related by 
a linear model
\[
b=A \bar{x} + z,
\]
where $\bar{x}$ is the sparse vector we would like to recover in
statistical applications, and~$z$ is a noise vector.
We assume that the noise level, measured by $\|A^T z\|_\infty$, 
is relatively small compared with the regularization parameter~$\lambda$.
This scenario is of great modern interest, and various properties of 
the solution $x^\star(\lambda)$ have been investigated 
\cite{CandesTao05,DonohoEladTemlyakov06,MeinshausenB06,Tropp06,ZhaoY06,
CandesTao07,ZhangHuang08,ZhangT09,BickelRT09,Koltchinskii09,vandeGeerB09,
Wainwright09}. 
In particular, it is known that under suitable conditions on~$A$ such as 
the \emph{restricted isometry property} (RIP), and as long as 
$\lambda \geq c \|A^T z\|_\infty$ (for some universal constant $c$),
one can obtain a recovery bound of the optimal form 
\begin{equation}
\|x^\star(\lambda)-\bar{x}\|_2^2 = O\left(\lambda^2 \|\bar{x}\|_0 \right) , 
\label{eq:opt-recovery}
\end{equation}
where $\|\bar{x}\|_0$ denotes the number of nonzero elements in~$\bar{x}$.
The constant in $O(\cdot)$ depends only on the so-called RIP condition
that we will discuss later on, 
and this bound achieves the optimal order of recovery.
Moreover, it is known that in this situation, the solution $x^\star(\lambda)$ 
is sparse \cite{ZhangHuang08}, and the sparsity of the solution is 
closely related to the recovery performance.

In this paper, we develop an efficient numerical method for solving the 
$\ell_1$-LS problem in the context of sparse recovery described above.
In particular, we focus on the case when $m<n$ 
(i.e., the linear system $Ax=b$ is underdetermined)
and the solution $x^\star(\lambda)$ is sparse 
(which requires the parameter~$\lambda$ to be sufficiently large).
Under such assumptions, our method has provable lower complexity 
than previous algorithms.

The $\ell_1$-LS problem~(\ref{eqn:l1-LS}) is closely related to the following
two constrained convex optimization problems:
\begin{equation}\label{eqn:Lasso}
\minimize_x \quad \|A x - b\|_2^2 
\quad \mbox{subject to} \quad \|x\|_1 \leq \Delta,
\end{equation}
known as the \emph{least absolute shrinkage and selection operator} (LASSO)
\cite{Tibshirani96}, and
\begin{equation}\label{eqn:l1-min}
\minimize_x \quad  \|x\|_1 
\quad \mbox{subject to} \quad \|A x - b\|_2^2\leq \varepsilon,
\end{equation}
where $\Delta$ and $\varepsilon$ are two nonnegative real parameters. 
These problems have the same solution as~(\ref{eqn:l1-LS}) for appropriate 
choices of the parameters $\lambda$, $\Delta$ and $\epsilon$.
However, other than in some special cases, the exact correspondence between 
these parameters are not known a priori.
Therefore, algorithms that are specific for solving one formulation 
may not be used directly for solving others.
Nevertheless, our method can be adapted to solve~(\ref{eqn:Lasso}) 
and~(\ref{eqn:l1-min}) efficiently, 
either by using an augmented Lagrangian approach \cite{YinOsherGD08},
or by using a root-finding procedure similar as the one given in
\cite{vandenBergFriedlander08}.

\subsection{Previous algorithms}
\label{sec:previous-algm}
There have been extensive research on numerical methods for solving
the problems~(\ref{eqn:l1-LS}), (\ref{eqn:Lasso}) and~(\ref{eqn:l1-min}).
A nice survey of major practical algorithms for sparse approximation 
appeared in \cite{TroppWright10}, and performance comparisons of various 
algorithms can be found in, e.g., 
\cite{WrightNF09,WenYinGZ10,BeckerBobinCandes11}.
Here we briefly summarize the computational complexities of several methods 
that are most relevant for solving the $\ell_1$-LS problem~(\ref{eqn:l1-LS}), 
in terms of finding an $\epsilon$-optimal solution 
(i.e., obtaining an objective value within~$\epsilon$ of the global minimum).

\emph{Interior-point methods} were among the first approaches used for
solving the $\ell_1$-LS problem \cite{ChenDS98,TurlachVW05,KKLBG07}.
The theoretical bound on their iteration complexity is 
$O\left(\sqrt{n}\log(1/\epsilon)\right)$, 
although their practical performance demonstrate much weaker dependence on~$n$.
The bottleneck of their performance is the computational cost per iteration.
For example, with an unstructured dense matrix~$A$, the standard approach of
solving the normal equation in each iteration with a direct method 
(Cholesky factorization) would cost $O(m^2 n)$ flops,
which is prohibitive for large-scale applications.
Therefore all customized solvers \cite{ChenDS98,TurlachVW05,KKLBG07}
use iterative methods (such as conjugate gradients) for solving the 
linear equations. 
These methods only require matrix-vector multiplications involving~$A$ 
and~$A^T$, and the computational cost per iteration can be $O(mn)$.
The cost can be further reduced if the matrix-vector multiplication can be
conducted more efficiently, 
e.g., $O(n\log n)$ if~$A$ is a partial Fourier matrix.

\emph{Proximal gradient methods} for solving the $\ell_1$-LS problem 
take the following basic form at each iteration $k=0,1,\ldots$
\begin{equation}\label{eqn:prox-grad-step}
x^{(k+1)} = \argmin_y \left\{ f(x^{(k)}) + \nabla\!f(x^{(k)})^T(y-x^{(k)}) 
+ \frac{L_k}{2}\|y-x^{(k)}\|_2^2 + \lambda \|y\|_1 \right\},
\end{equation}
where we used the shorthand $f(x)=(1/2)\|Ax-b\|_2^2$, 
and $L_k$ is a parameter chosen at each iteration
(e.g., using a line-search procedure).
The minimization problem in~(\ref{eqn:prox-grad-step}) has a 
closed-form solution
\begin{equation}\label{eqn:IST}
x^{(k+1)} = \SoftThresholding 
\left( x^{(k)}-\frac{1}{L_k}\nabla\!f(x^{(k)})\, ,~\frac{\lambda}{L_k}\right),
\end{equation}
where $\SoftThresholding:\reals^n \times \reals^+ \to \reals^n$ is the 
well-known \emph{soft-thresholding} operator, defined as
\begin{equation}\label{eqn:soft}
(\SoftThresholding(x,\alpha))_i 
= \sgn(x_i) \max \left\{ |x_i| - \alpha, ~0 \right\},
\quad i=1,\ldots,n.
\end{equation}
Iterative methods that use the update rule~(\ref{eqn:IST}) include
\cite{DaubechiesDD04,CombettesWajs05,Nesterov07composite, 
HaleYinZhang08,WrightNF09}.
Their major computational effort per iteration is to form the gradient
$\nabla\! f(x)=A^T(Ax-b)$, which costs $O(mn)$ flops for a generic
dense matrix~$A$.
With appropriate choices of the parameters~$L_k$, the proximal-gradient 
method~(\ref{eqn:prox-grad-step}) has an iteration complexity $O(1/\epsilon)$.
  
Indeed, the iteration complexity $O(\log(1/\epsilon))$ can be established 
for~(\ref{eqn:prox-grad-step}) if $m\geq n$ and $A$ has full column rank, 
since in this case the objective function in~(\ref{eqn:l1-LS}) is 
strongly convex \cite{Nesterov07composite}.
Unfortunately this result is not applicable to the case $m<n$.
Nevertheless, when the solution $x^\star(\lambda)$ is sparse and the active 
submatrix is well conditioned (e.g., when~$A$ has RIP),
local linear convergence can be established \cite{LuoTseng92,HaleYinZhang08}, 
and fast convergence in the final stage of the algorithm has also been 
observed \cite{Nesterov07composite,HaleYinZhang08,WrightNF09}.

\emph{Variations and extensions of the proximal gradient method}
have been proposed to speed up the convergence in practice; see, e.g., 
\cite{BioucasFigueiredo07,WrightNF09,WenYinGZ10}.
Nesterov's optimal gradient methods for minimizing smooth convex functions
\cite{Nesterov83,Nesterov04book,Nesterov05} 
have also been extended to minimize composite objective functions such as 
in the $\ell_1$-LS problem 
\cite{Nesterov07composite,Tseng08,BeckTeboulle09,BeckerBobinCandes11}.
These accelerated methods have the iteration complexity $O(1/\sqrt{\epsilon})$.
They typically generate two or three concurrent sequences of iterates,
but their computational cost per iteration is still $O(mn)$, 
which is the same as simple gradient methods. 

\emph{Exact homotopy path-following methods} were developed in the statistics
literature to compute the complete \textsc{LASSO} path when varying
the regularization parameter~$\lambda$ from large to small
\cite{OsbornePT00a, OsbornePT00b,EfronHJT04}.
These methods exploit the piece-wise linearity of the solution as a function
of~$\lambda$, and identify the next breakpoint along the solution path
by examining the optimality conditions
(also called \emph{active set} or \emph{pivoting} method in optimization).
With efficient numerical implementations (using updating or downdating of 
submatrix factorizations), the computational cost at each break point is 
$O(mn+ms^2)$, where~$s$ is the number of nonzeros in the solution at the
breakpoint. Such methods can be quite efficient if~$s$ is small.
However, in general, there is no convergence result bounding the number of 
breakpoints for this class of methods
(for some special cases, the number of breakpoints is the same as the number 
of nonzeros in the solution \cite{DonohoTsaig08}).

\emph{Greedy algorithms} such as \emph{orthogonal matching pursuit (OMP)} 
are also very popular for sparse recovery applications 
(e.g., \cite{DavisMallatA97,Tropp04,NeedellTropp09}).
However, they are not designed to solve any of the optimization 
problems~(\ref{eqn:l1-LS}), (\ref{eqn:Lasso}) or~(\ref{eqn:l1-min}).
Their connections with exact homotopy methods are analyzed in 
\cite{DonohoTsaig08}.

\subsection{Proposed approach and contributions}
We consider an \emph{approximate} homotopy continuation method,
where the key idea is to solve~(\ref{eqn:l1-LS}) with a large regularization
parameter $\lambda$ first, and then gradually decreases $\lambda$ 
until the target regularization is reached. 
For each fixed $\lambda$, we employ a proximal gradient method 
of the form~(\ref{eqn:prox-grad-step}) to solve (\ref{eqn:l1-LS}) 
up to an adequate precision (to be specified later), 
and then use this approximate solution to serve as the initial point for 
the next value of $\lambda$.
We call the resulting method \emph{proximal-gradient homotopy} (PGH) method.

This is not a new idea. 
Approximate homotopy continuation methods that use proximal gradient methods
for solving each stage (with a fixed value of~$\lambda$) have been studied in, 
e.g., \cite{HaleYinZhang08,WrightNF09,WenYinGZ10}, and superior empirical 
performance have been reported when the solution is sparse.
However, there has been no effective theoretical analysis for their overall
iteration complexity.
As a result, some important algorithmic choices are mostly based on 
heuristics and ad hoc factors.
More specifically, how do we choose the sequence of decreasing values 
for~$\lambda$? and how accurate should we solve the problem~(\ref{eqn:l1-LS})
for each value in this sequence? 

In this paper, we present a PGH method that has provable low iteration 
complexity, along with the following specific algorithmic choices:
\begin{itemize}
\item We use a decreasing geometric sequence for the values of~$\lambda$.
That is, we choose a $\lambda_0$ and a parameter $\eta\in(0,1)$, 
and let $\lambda_K=\eta^K\lambda_0$ for $K=1,2,\ldots$ 
until the target value is reached.
\item We choose a parameter $\delta\in(0,1)$ and solve 
problem~(\ref{eqn:l1-LS}) for each~$\lambda_K$ with a proportional 
precision~$\delta\lambda_K$ (in terms of violating the optimality condition),
except that for the final target value of~$\lambda$, 
we reach the absolute precision $\epsilon$. 
\item We use Nesterov's adaptive line-search strategy in 
\cite{Nesterov07composite} to choose the parameters~$L_k$ 
in the proximal gradient method~(\ref{eqn:prox-grad-step}).
\end{itemize}
Under the assumptions that the target value of~$\lambda$ is sufficiently large
(such that the final solution is sparse)
and the matrix~$A$ satisfies a RIP-like condition, our PGH method exhibits
geometric convergence at each stage, and the overall iteration complexity
is $O(\log(1/\epsilon))$. 
The constant in $O(\cdot)$ depends on the RIP-like condition.
Moreover, it is sufficient to choose $\lambda \geq c \|A^T z\|_\infty$ 
(for some universal constant $c$), which implies that the solution satisfies 
a recovery bound of the optimal form (\ref{eq:opt-recovery}).
Since each iteration of the proximal gradient method cost $O(mn)$ flops,
the overall computational complexity is $O(mn\log(1/\epsilon))$, 
implying global geometric rate of convergence.

The low iteration complexity of our PGH method is achieved by actively
exploiting the fast local linear convergence of the standard proximal
gradient method when the solution $x^\star(\lambda)$ is sparse 
\cite{LuoTseng92,HaleYinZhang08}.
Using the homotopy continuation strategy, the proximal gradient method
at each stage always starts with a point that is close to its solution.
Moreover, by choosing appropriate parameters~$\eta$ and~$\delta$ in our method,
we ensure that all iterates along the solution path 
(i.e., not only the final points) at each stage are sufficiently sparse. 
Under a RIP-like assumption on~$A$, this implies that along the homotopy path, 
the objective function in~(\ref{eqn:l1-LS}) is effectively strongly convex, 
and hence global geometric rate can be established using Nesterov's analysis 
\cite{Nesterov07composite}.

The advantage of our method over the exact homotopy path-following approach 
(\cite{OsbornePT00a, OsbornePT00b,EfronHJT04})
is that there is no need to keep track of all breakpoints.
In fact, for large-scale problems, the total number of proximal gradient steps 
in our method can be much smaller than the number of nonzeros 
in the target solution, which is the minimum number of breakpoints
the exact homotopy methods have to compute.
This phenomenon is predicted by our low iteration complexity, 
and also confirmed in our empirical studies. 

Compared with interior-point methods (IPMs), our methods has a similar 
iteration complexity (actually better in terms of theoretical bounds), 
and computationally can be much more efficient for each iteration.
The approximate homotopy strategy used in this paper is also analogous to 
the long-step path-following IPMs (e.g., \cite{Nesterov96}),
in the sense that the least-squares problem becomes better conditioned 
near the regularization path (cf.\ \emph{central path} in IPMs).
However, our results only hold for problems with provable sparse solutions,
and the parameters~$\eta$ and~$\delta$ depends on the problem data~$A$ 
and the regularization parameter~$\lambda$.
In contrast, the performance of interior-point methods is insensitive
to the sparsity of the solution or the regularization parameter.

As an important special case, our results can be immediately applied to 
noise-free compressed-sensing applications. 
Consider the \emph{basis pursuit (BP)} problem
\begin{equation}\label{eqn:BP}
\minimize \quad \|x\|_1 \quad \mbox{subject to} \quad A x =  b,
\end{equation}
which is a special case of~(\ref{eqn:l1-min}) with $\varepsilon=0$.
Its solution can be obtained by running our PGH method on the $\ell_1$-LS 
problem~(\ref{eqn:l1-LS}) with $\lambda\to 0$.
In terms of satisfying the condition $\lambda > c\, \|Az\|_\infty$, 
any $\lambda>0$ is sufficiently large in the noise-free case because $z=0$.
Therefore, the global geometric convergence of the PGH method for BP is just 
a special case of the more general result for (\ref{eqn:l1-LS}) developed 
in this paper.

It is also worth mentioning that variants of the proximal gradient 
method~(\ref{eqn:prox-grad-step}) can be directly applied to the constrained 
LASSO formulation~(\ref{eqn:Lasso}).
Moreover, under suitable conditions and when the parameter~$\Delta$ is set
to nearly equal to $\|\bar{x}\|_1$, geometric convergence 
\emph{away} from the optimal solution can be established \cite{AgNeWa11}.
However, for sparse recovery applications, such a result is less 
satisfactory than the homotopy approach we analyze in this paper 
due to the requirement of estimating $\|\bar{x}\|_1$ 
--- which is extremely difficult to determine efficiently in practice even
for the simple noise-free case of basis pursuit.
The proof techniques are also different, and the analysis of 
geometric convergence for PGH is more difficult than that of \cite{AgNeWa11}, 
because we have to demonstrate sparsity of all the intermediate solutions 
in the proximal gradient steps along the homotopy path. 
A significantly simpler argument can be used in \cite{AgNeWa11}, 
if the extra knowledge of $\|\bar{x}\|_1$ is known a priori.

\subsection{Outline of the paper}
In Section~\ref{sec:preliminaries}, we review some preliminaries
that are necessary for developing our method and its convergence analysis.
In Section~\ref{sec:PGH-method}, we present our proximal-gradient homotopy 
(PGH) method, and state the assumptions and the main convergence results.
Section~\ref{sec:analysis} is devoted to the proofs of our convergence results. 
We present numerical experiments in Section~\ref{sec:experiments} to
support our theoretical analysis, 
and conclude in Section~\ref{sec:discussions} with some further discussions.

\section{Preliminaries and notations}
\label{sec:preliminaries}

In this section, we first introduce composite gradient mapping and some
of its key properties developed in \cite{Nesterov07composite}.
Then we describe Nesterov's proximal gradient method with adaptive line search,
which we will use to solve the $\ell_1$-LS problem at each stage of our PGH 
method.
Finally we discuss the restricted eigenvalue conditions that allow us
to show the local linear convergence of Nesterov's algorithm.

\subsection{Composite gradient mapping}
Consider the following optimization problem with \emph{composite} 
objective function:
\begin{equation}\label{eqn:composite-min}
\minimize_x \quad \left\{ \phi(x) \triangleq f(x) + \Psi(x) \right\},
\end{equation}
where the function $f$ is convex and differentiable, 
and~$\Psi$ is closed and convex on $\reals^n$.
The optimality condition of~(\ref{eqn:composite-min}) states that 
$x^\star$ is a solution if and only if there exists 
$\xi\in\partial\Psi(x^\star)$ such that
\[
\nabla\! f(x^\star)  + \xi = 0
\]
(see, e.g., \cite[Section~27]{Rockafellar70}).
Therefore, a good measure of accuracy for any~$x$ as an approximate
solution is the quantity
\begin{equation}\label{eqn:opt-residue}
\omega(x) \triangleq \min_{\xi\in\partial\Psi(x)} \|\nabla\! f(x)+\xi\|_\infty.
\end{equation}
We call $\omega(x)$ the \emph{optimality residue} of~$x$.
We will use it in the stopping criterion of the proximal gradient method.

Composite gradient mapping was introduced by Nesterov in
\cite{Nesterov07composite}.
For any fixed point~$y$ and a given constant $L>0$, 
we define a local model of~$\phi(x)$ around~$y$
using a quadratic approximation of~$f$ but keeping $\Psi$ intact:
\[
\psi_L(y;x) = f(y) + \nabla\! f(y)^T (x-y) 
+ \frac{L}{2}\|x-y\|_2^2 + \Psi(x).
\]
Let
\begin{equation}\label{eqn:prox-grad}
T_L(y) = \argmin_x ~\psi_L(y; x).
\end{equation}
Then the \emph{composite gradient mapping} of~$f$ at~$y$ is defined as
\[
g_L(y) = L(y-T_L(y)).
\]
In the case $\Psi(x)=0$, it is easy to verify that $g_L(y)=\nabla\!f(y)$ 
for any $L>0$, and $1/L$ can be considered as the step-size from 
$y$ to $T_L(y)$ along the direction $-g_L(y)$.
The following property of composite gradient mapping was shown 
in \cite[Theorem~2]{Nesterov07composite}:
\begin{lemma}\label{lem:cgm-decrement}
For any $L>0$,
\[
\psi_L (y;T_L(y) ) \leq \phi(y) - \frac{1}{2L} \|g_L(y)\|_2^2 .
\]
\end{lemma}

The function~$f$ has Lipschitz continuous gradient if there exists
a constant $L_f$ such that
\[
\left\| \nabla\! f(x) - \nabla\!f(y) \right\|_2 \leq L_f \|x-y\|_2,
\quad \forall\, x,y\in\reals^n .
\]
A direct consequence of having Lipschitz continuous gradient is 
the following inequality (see, e.g., \cite[Theorem~2.1.5]{Nesterov04book}):
\begin{equation}\label{eqn:lipschitz-ub}
f(y) \leq f(x) + \langle \nabla\!f(x), y-x \rangle + \frac{L_f}{2}\|y-x\|_2^2,
\quad \forall\, x,y\in\reals^n .
\end{equation}
For such functions, we can measure how close $T_L(y)$ is from satisfying the 
optimality condition by using the norm of the composite gradient mapping at~$y$.
\begin{lemma}\label{lem:kkt-grad-mapping}
If~$f$ has Lipschitz continuous gradients with Lipschitz constant~$L_f$, then
\[
\omega(T_L(y))
\leq \left(1+\frac{S_L(y)}{L}\right) \| g_L(y) \|_2
\leq \left(1+\frac{L_f}{L} \right) \|g_L(y)\|_2
\]
where $S_L(y)$ is a local Lipschitz constant defined as
\[
S_L(y) = \frac{ \|\nabla\!f(T_L(y)) - \nabla\! f(y) \|_2}{\|T_L(y)-y\|_2} .
\]
\end{lemma}
\begin{proof}
Let $D\phi(x)[u]$ denote the directional derivative of~$\phi$ at~$x$
along the direction~$u$, i.e, 
\[
D\phi(x)[u] = \lim_{\alpha\downarrow 0} \frac{1}{\alpha}
\bigl( \phi(x + \alpha u) - \phi(x) \bigr).
\]
Corollary~1 in \cite{Nesterov07composite} states that for 
any $u\in\reals^n$ with $\|u\|_2=1$, the following inequality holds:
\[
D\phi(T_L(y))[u]  \geq - \left(1+\frac{S_L(y)}{L}\right) \|g_L(y)\|_2.
\]
In addition, it is shown in \cite{Nesterov07composite}
that for any $x\in\reals^n$,
\[
\min_{\xi\in\partial\Psi(x)} \|\nabla\!f(x)+\xi\|_2
= - \min_{\|u\|_2=1} D\phi(x)[u] .
\]
(See \cite[Section~2]{Nesterov07composite}.)
Therefore, we have
\[
\omega(T_L(y)) 
\leq \min_{\xi\in\partial\Psi(T_L(y))} \|\nabla\!f(T_L(y))+\xi\|_2
\leq \left(1+\frac{S_L(y)}{L}\right) \| g_L(y) \|_2.
\]
The last desired inequality follows from the fact $S_L(y) \leq L_f$.
\end{proof}

In this paper, we use the following notations to simplify presentation:
\begin{eqnarray*}
f(x) &=& \frac{1}{2}\|A x - b\|_2^2\\
\phi_\lambda(x) &=& f(x) + \lambda \|x\|_1.
\end{eqnarray*}
Correspondingly, we add the subscript $\lambda$ in specifying the 
composite gradient mapping:
\begin{eqnarray*}
\psi_{\lambda,L}(y; x) 
&=& f(y) + \nabla\! f(y)^T (x-y) 
    + \frac{L}{2}\|x-y\|_2^2 + \lambda \|x\|_1 \\
T_{\lambda,L}(y) &=& \argmin_x ~\psi_{\lambda,L}(y;x) \\
g_{\lambda,L}(y) &=& L \bigl( y-T_{\lambda,L}(y) \bigr) \\
\omega_\lambda(x) &=& \min_{\xi\in\partial\|x\|_1} 
                      \|\nabla\!f(x)+\lambda\xi\|_\infty.
\end{eqnarray*}
We call the process of computing $T_L(y)$ a proximal gradient step.
For the $\ell_1$-LS problem, 
$T_{\lambda,L}(x)$ has the closed-form solution given in~(\ref{eqn:IST}).
Given the gradient $\nabla\!f(x)$, the optimality residue
$\omega_\lambda(x)$ can be easily computed with $O(n)$ flops.

\subsection{Nesterov's gradient method with adaptive line-search}
\label{sec:prox-grad}

\begin{algorithm}[t]
\DontPrintSemicolon
\SetKwInOut{Input}{input}
\Input{$\lambda>0$, $x\in\reals^n$, $L>0$}
\textbf{parameter:} $\gammaInc>1$\;
  \Repeat{$\phi_\lambda(x^+) <= \psi_{\lambda,L}(x;x^+)$}{
    $x^+ \gets T_{\lambda,L}(x)$\;\vspace{0.2ex}
    \lIf{$\phi_\lambda(x^+) > \psi_{\lambda,L}(x;x^+)$}{
         $L\gets L\gammaInc$}
    \vspace{0.5ex}
  }
  $M \gets L$\;
\Return{$\{x^+, M\}$}\;
\caption{$\{x^+,M\} \gets \texttt{LineSearch} (\lambda, x, L)$}
\label{alg:line-search}
\end{algorithm}

\begin{algorithm}[t]
\DontPrintSemicolon
\SetKwInOut{Input}{input}
\Input{$\lambda>0$, $\hat{\epsilon}>0$, $x^{(0)}\in\reals^n$, 
  $L_0\geq \Lmin$}
\textbf{parameters:} $\Lmin>0$, $\gammaDec\geq 1$\;
\Repeat(for $k=0,1,2,\ldots$){$\omega_\lambda(x^{(k+1)}) \leq \hat\epsilon$}{
  $\{x^{(k+1)},M_k\}\gets \texttt{LineSearch} 
     (\lambda, x^{(k)}, L_k )$\;
  $L_{k+1} \gets \max\{\Lmin, M_k/\gammaDec\}$\;
}
$\hat x \gets x^{(k+1)}$\;
$\hat M \gets M_k$\;
\Return{$\{\hat x, \hat M\}$}\;
\caption{$\{\hat x, \hat M\} \gets \texttt{ProxGrad} 
  (\lambda, \hat{\epsilon}, x^{(0)}, L_0 )$}
\label{alg:prox-grad}
\end{algorithm}

With the machinery of composite gradient mapping, Nesterov developed 
several variants of proximal gradient methods in \cite{Nesterov07composite}.
We use the non-accelerated primal-gradient version described in
Algorithms~\ref{alg:line-search} and~\ref{alg:prox-grad}, which correspond
to~(3.1) and~(3.2) in \cite{Nesterov07composite}, respectively.
To use this algorithm, we need to first choose an initial optimistic estimate
$\Lmin$ for the Lipschitz constant $L_f$:
\[
0 < \Lmin \leq L_f,
\]
and two adjustment parameters $\gammaDec\geq 1$ and $\gammaInc>1$. 
A key feature of this algorithm is the adaptive line search: 
it always tries to use a smaller Lipschitz constant first at each iteration.

Each iteration of the proximal gradient method generates the next iterate
in the form of
\[
x^{(k+1)} = T_{\lambda,M_k}(x^{(k)}) ,
\]
where~$M_k$ is chosen by the line search procedure 
in Algorithm~(\ref{alg:line-search}).
The line search procedure starts with an estimated Lipschitz constant $L_k$,
and increases its value by the factor $\gammaInc$ 
until the stopping criteria is satisfied.
The stopping criteria for line search ensures 
\begin{eqnarray}
\phi_\lambda ( x^{(k+1)} ) 
&\leq& \psi_{\lambda,M_k} \left( x^{(k)},x^{(k+1)} \right) 
~=~ \psi_{\lambda,M_k} \left( x^{(k)},T_{\lambda,M_k}(x^{(k)}) \right) 
    \nonumber\\
&\leq& \phi_\lambda(x^{(k)}) 
- \frac{1}{2M_k} \bigl\|g_{\lambda,M_k}(x^{(k)}) \bigr\|_2^2, 
\label{eqn:monotone-decrease}
\end{eqnarray}
where the last inequality follows from Lemma~\ref{lem:cgm-decrement}.
Therefore, we have the objective value $\phi_\lambda(x^{(k)})$ decrease 
monotonically with~$k$, unless the gradient mapping 
$g_{\lambda,M_k}(x^{(k)})=0$.
In the latter case, according to Lemma~\ref{lem:kkt-grad-mapping},
$x^{(k+1)}$ is an optimal solution.

The only difference between Algorithm~\ref{alg:prox-grad} and Nesterov's
gradient method \cite[(3.2)]{Nesterov07composite} is that 
Algorithm~\ref{alg:prox-grad} has an explicit stopping criterion.
This stopping criterion is based on the optimality residue 
$\omega_\lambda(x^{(k+1)})$ being small.
For the~$\ell_1$-LS problem, it can be computed with additional $O(n)$ flops 
given the gradient $\nabla\! f(x)$.
For other problems, depending on the form of~$\Psi$, 
this residue may be hard to compute.
But we can always use the alternative stopping criterion
\[
\bigl\| g_{\lambda,M_k}(x^{(k)}) \bigr\|_2 \leq \hat{\epsilon}.
\]
According to Lemma~\ref{lem:kkt-grad-mapping}, these two measures may differ
by a factor $(1+S_{M_k}(x^{(k+1)})/M_k)$. 
So the precision $\hat{\epsilon}$ may need to be reduced by a similar factor.

Since~$f$ has Lipschitz constant~$L_f$, the inequality~(\ref{eqn:lipschitz-ub})
implies that the line search procedure is guaranteed to terminate 
if $L\geq L_f$. 
Therefore, we have
\begin{equation}\label{eqn:lipsch-bounds}
\Lmin \leq L_k \leq M_k < \gammaInc L_f .
\end{equation}
Although there is no explicit bound on the number of repetitions in the 
line search procedure, Nesterov showed that the total number of 
line searches cannot be too big. 
More specifically, let $N_k$ be the number of operations
$x^+\gets T_{\lambda,L}(x)$ after~$k$ iterations 
in Algorithm~\ref{alg:prox-grad}.
Lemma~3 in \cite{Nesterov07composite} showed that 
\[
N_k ~\leq~ \left(1+\frac{\ln\gammaDec}{\ln\gammaInc}\right)
(k+1) + \frac{1}{\ln\gammaInc} \max\left\{
\ln\frac{\gammaInc L_f}{\gammaDec \Lmin},0\right\}.
\]
For example, if we choose $\gammaInc=\gammaDec=2$, then
\begin{equation}\label{eqn:line-search-bound}
N_k \leq 2(k+1) + \log_2\frac{L_f}{\Lmin}.
\end{equation}

Nesterov established the following iteration complexities of
Algorithm~\ref{alg:prox-grad} for finding an $\epsilon$-optimal solution
of the problem~(\ref{eqn:composite-min}):
\begin{itemize}
\item
If $\phi_\lambda$ is convex but not strongly convex, then the convergence
is sublinear, with an iteration complexity $O(1/\epsilon)$ 
\cite[Theorem~4]{Nesterov07composite};
\item
If $\phi_\lambda$ is strongly convex, then the convergence is geometric, 
with an iteration complexity $O(\log(1/\epsilon))$ 
\cite[Theorem~5]{Nesterov07composite}.
\end{itemize}
A nice property of this algorithm is that we do not need to know a priori
if the objective function is strongly convex or not. 
It will automatically exploit the strong convexity whenever it holds.
The algorithm is the same for both cases.

For our interested case $m<n$, the objective function in 
Problem~(\ref{eqn:l1-LS}) is not strongly convex.
Therefore, if we directly use Algorithm~\ref{alg:prox-grad} to solve
this problem, we can only get the $O(1/\epsilon)$ iteration complexity 
(even though fast local linear convergence was observed in 
\cite{Nesterov07composite} when the solution is sparse).
Nevertheless, as explained in the introduction, we can use a homotopy 
continuation strategy to enforce that all iterates along the solution path 
are sufficiently sparse.
Under a RIP-like assumption on~$A$, this implies that the objective function 
is effectively strongly convex along the homotopy path, 
and hence global geometric rate can be established using Nesterov's analysis.
Next we explain conditions that characterize restricted strong convexity
for sparse vectors.

\subsection{Restricted eigenvalue conditions}
\label{sec:restr-eig-cond}

We first define some standard notations for sparse recovery.
For a vector $x\in\reals^n$, let
\[
\supp(x) = \{j: x_j \neq 0\},  \qquad
\|x\|_0 = |\supp(x)| .
\]
Throughout the paper, we denote $\supp(\bar{x})$ by $\bar{S}$, 
and use $\bar{S}^c$ for its complement.
We use the notations $x_{\bar{S}}$ and $x_{\bar{S}^c}$ to denote the 
restrictions of a vector $x$ to the coordinates indexed by $\bar{S}$
and $\bar{S}^c$, respectively.

Various conditions for sparse recovery have appeared in the literature.
The most well-known of such conditions is the 
\emph{restricted isometry property} (RIP) introduced in \cite{CandesTao05}. 
In this paper, we analyze the numerical solution of the $\ell_1$-LS problem
under a slight generalization, which we refer to as 
\emph{restricted eigenvalue condition}.
\begin{definition} \label{def:sparse-eigenvalue}
Given an integer $s>0$, we say that $A$ satisfies the restricted eigenvalue 
condition at sparsity level $s$ if there exists positive constants 
$\rho_-(A,s)$ and $\rho_+(A,s)$ such that
\begin{eqnarray*}
\rho_+(A,s) &=& \sup \left\{ \frac{x^T A^T A x}{x^T x} 
    : x \neq 0, ~\|x \|_0 \leq s \right\} , \\
\rho_-(A,s) &=& \inf \left\{ \frac{x^T A^T A x}{x^T x} 
    : x \neq 0, ~\|x \|_0 \leq s \right\} .
\end{eqnarray*}
\end{definition}

Note that a matrix $A$ satisfies the original definition of restricted isometry property 
with RIP constant $\nu$ at sparsity level $s$ 
if and only if $\rho_+(A,s) \leq 1+\nu$ and $\rho_-(A,s) \geq 1-\nu$.
More generally, the strong convexity of the objective function 
in~(\ref{eqn:l1-LS}), namely~$\phi_\lambda(x)$,
is equivalent to $\rho_-(A,n)>0$.
However, since we are interested in the situation of $m < n$, 
which implies that $\rho_-(A,n)=0$, we know that~$\phi_\lambda$
is not strongly convex. 
Nevertheless, for $s<m$, it is still possible that the condition 
$\rho_-(A,s)>0$ holds. 
This means that if both~$x$ and~$y$ are sparse vectors,
then~$\phi_\lambda$ is strongly convex along the line segment that 
connects~$x$ and~$y$. 
Moreover, the inequality that characterize the smoothness of the function,
namely~(\ref{eqn:lipschitz-ub}), could use a much smaller restricted
Lipschitz constant instead of the global constant~$L_f=\rho_+(A, n)$.
More precisely, we have the following lemma.

\begin{lemma}
Let $f(x)=(1/2)\|A x-b\|_2^2$.
Suppose~$x$ and~$y$ are two sparse vectors such that
\[
| \supp(x) \cup \supp(y) | \leq s
\]
for some integer $s<m$. Then the following two inequalities hold:
\begin{eqnarray}
f(y) &\leq& f(x) + \langle \nabla\!f(x), y-x \rangle + 
\frac{\rho_+(A,s)}{2} \|y-x\|_2^2, 
\label{eqn:restr-smoothness}  \\
f(y) &\geq& f(x) + \langle \nabla\!f(x), y-x \rangle + 
\frac{\rho_-(A,s)}{2} \|y-x\|_2^2.
\label{eqn:restr-strong-convex}
\end{eqnarray}
\end{lemma}
\begin{proof}
For any $x,y\in\reals^n$, it is straightforward to verify that
if $f(x)=(1/2)\|A x-b\|_2^2$, then
\[
f(y) - f(x) - \langle \nabla\!f(x), y-x \rangle ~=~ \frac{1}{2}\|A(y-x)\|_2^2.
\]
Since the assumption $| \supp(x) \cup \supp(y) | \leq s$ 
implies $\|y-x\|_0 \leq s$, we use the definition of 
restricted eigenvalues to conclude
\[
\rho_-(A,s) \|y-x\|_2^2 ~\leq~ \|A(y-x)\|_2^2 ~\leq~ \rho_+(A, s) \|y-x\|_2^2.
\]
These lead to the two desired inequalities.
\end{proof}

The inequality~(\ref{eqn:restr-smoothness}) represents
\emph{restricted smoothness}, and~(\ref{eqn:restr-strong-convex}) represents 
\emph{restricted strong convexity}.
A key feature of our PGH method is that sparsity along the whole solution path
can be enforced. Therefore the objective function in~(\ref{eqn:l1-LS}) 
becomes strongly convex along the solution path if the sparse eigenvalues in 
Definition~\ref{def:sparse-eigenvalue} are well behaved 
(i.e., they grow slowly when~$s$ is increased).
In such a situation, the PGH method exhibits geometric convergence along
the homotopy path, and the convergence rate depends on a
\emph{restricted condition number}, defined as
\begin{equation}\label{eqn:restr-cond}
\kappa(A,s) = \frac{\rho_+(A,s)}{\rho_-(A,s)}. 
\end{equation}
In particular, if the matrix $A$ has RIP constant $\nu$ at sparsity level $s$, then
$\kappa(A,s) \leq (1+\nu)/(1-\nu)$.

\section{A proximal-gradient homotopy method}
\label{sec:PGH-method}

\begin{algorithm}[t]
\DontPrintSemicolon
\SetKwInOut{Input}{input}
\Input{$A\in\reals^{m\times n}$, $b\in\reals^n$, $\lambda_\mathrm{tgt}>0$, 
       $\epsilon>0$, $\Lmin>0 $}
\textbf{parameters:} $\eta\in(0,1)$, $\delta\in(0,1)$\;
\textbf{initialize:} $\lambda_0 \gets \|A^T b\|_\infty$, 
   ~$\hat x^{(0)} \gets 0$, ~$\hat{M}_0 \gets \Lmin$\;
$N \gets \left\lfloor\ln\!\left(\lambda_0/\lambda_\mathrm{tgt}\right) / 
   \ln(1/\eta)\right\rfloor$\;
\For{$K=0,1,2,\ldots,N-1$}{
  $\lambda_{K+1} \gets \eta \lambda_K$\;
  $\hat{\epsilon}_{K+1} \gets \delta\lambda_{K+1}$\;
  $\{\hat x^{(K+1)},\hat M_{K+1}\}\gets \texttt{ProxGrad} 
   \bigl(\lambda_{K+1}, \hat{\epsilon}_{K+1}, \hat x^{(K)}, \hat M_K \bigr)$\;
}
$\{\hat{x}^\mathrm{(tgt)},\hat M_\mathrm{tgt}\}\gets \texttt{ProxGrad} 
  \bigl( \lambda_\mathrm{tgt}, \epsilon, \hat x^{(N)}, \hat M_N \bigr)$\;
\Return{$\hat{x}^\mathrm{(tgt)}$}\;
\caption{$\hat{x}^\mathrm{(tgt)} \gets \texttt{Homotopy} 
          (A, b, \lambda_\mathrm{tgt}, \epsilon, \Lmin)$}
\label{alg:homotopy}
\end{algorithm}

The key idea of the proximal-gradient homotopy (PGH) method
is to solve~(\ref{eqn:l1-LS}) with a large regularization
parameter $\lambda_0$ first, and then gradually decreases $\lambda$ 
until the target regularization is reached. 
For each fixed $\lambda$, we employ Nesterov's proximal-gradient method 
described in Algorithms~\ref{alg:line-search} and~\ref{alg:prox-grad},
to solve problem~(\ref{eqn:l1-LS}) up to an adequate precision.
Then we use this approximate solution to warm start the PG method for 
the next value of $\lambda$.

Our proposed PGH method is listed as Algorithm~\ref{alg:homotopy}.
To make the presentation more clear, 
we use $\lambda_\mathrm{tgt}$ to denote the target regularization parameter.
The method starts with
\[
\lambda_0 = \|A^T b\|_\infty,
\]
since this is the smallest value for~$\lambda$ such that the $\ell_1$-LS
problem has the trivial solution~$0$ (by examining the optimality condition).
Our method has two parameters $\eta\in(0,1)$ and $\delta\in(0,1)$.
They control the algorithm as follows:
\begin{itemize}
\item The sequence of values for the regularization parameter is determined as
$\lambda_K=\eta^K\lambda_0$ for $K=1,2,\ldots$,
until the target value $\lambda_\mathrm{tgt}$ is reached.
\item For each~$\lambda_K$ except $\lambda_\mathrm{tgt}$,
we solve problem~(\ref{eqn:l1-LS}) with a proportional 
precision~$\delta\lambda_K$.
For the last stage with $\lambda_\mathrm{tgt}$, 
we solve to the absolute precision $\epsilon$. 
\end{itemize}

As discussed in the introduction, sparse recovery by solving the 
$\ell_1$-LS problem requires two types of conditions:
the regularization parameter~$\lambda$ is relatively 
large compared with the noise level,
and the matrix~$A$ satisfies certain RIP or restricted eigenvalue condition.
It turns out that such conditions are also sufficient for fast convergence of
our PGH method.
More precisely, we have the following assumption:
\begin{assumption}\label{asmp:mixed-re}
Suppose $b=A\bar{x}+z$. 
Let $\bar{S}=\supp(\bar{x})$ and $\bar{s}=|\bar{S}|$.
There exist $\gamma>0$ and $\delta'\in(0,1)$ such that $\gamma > (1+\delta')/(1-\delta')$ and
\begin{equation}\label{eqn:big-lambda}
\lambda_\mathrm{tgt} ~\geq~ 
\max\left\{4, ~\frac{\gamma+1}{(1-\delta')\gamma-(1+\delta')} \right\} 
\|A^T z\|_\infty . 
\end{equation}
Moreover, there exists an integer $\tilde{s}$ such that 
$\rho_-(A, \bar{s}+2\tilde{s})>0$ and
\begin{equation}\label{eqn:mixed-re}
\tilde{s} ~>~ 
\frac{ 16 \bigl( \gammaInc 
 \rho_+(A, \bar{s}+2\tilde{s})+2\rho_+(A, \tilde{s})\bigr)}
{\rho_-(A, \bar{s}+\tilde{s})} (1+\gamma)\bar{s}.
\end{equation}
We also assume that $\Lmin \leq \gammaInc  \rho_+(A, \bar{s}+2\tilde{s})$.
\end{assumption}

According to \cite{ZhangHuang08}, the above assumption implies that the 
solution $x^\star(\lambda)$ of (\ref{eqn:l1-LS}) is sparse whenever 
$\lambda \geq \lambda_\mathrm{tgt}$; more specifically,
$\|x^\star(\lambda)_{\bar{S}^c}\|_0 \leq \tilde{s}$ 
(here $\bar{S}^c$ denotes the complement of the support set~$\bar{S}$).
In this paper, we will show that by choosing the parameters~$\eta$ 
and~$\delta$ in Algorithm~\ref{alg:homotopy} appropriately, 
these conditions also imply that all iterates along the solution path
are sparse. 
Our proof employs a similar argument as that of \cite{ZhangHuang08}. 
Before stating the main convergence results, we make some further remarks
on Assumption~\ref{asmp:mixed-re}.
\begin{itemize}
\item
The condition~(\ref{eqn:big-lambda}) states that the~$\lambda$ must be
sufficiently large to dominate the noise.
Such a condition is adequate for sparse recovery applications because 
recovery performance given in (\ref{eq:opt-recovery}) achieves optimal 
error bound under stochastic noise model by picking $\lambda$ of the order 
$\|A^T z\|_\infty$
\cite{CandesTao07,ZhangHuang08,ZhangT09,BickelRT09,Koltchinskii09,
vandeGeerB09,Wainwright09}. 
Moreover, it is also necessary because when $\lambda$ is smaller than
the noise level, the solution $x^\star(\lambda)$ will not be sparse anymore, 
which defeats the practical purpose of using $\ell_1$ regularization.
\item
The existence of $\tilde{s}$ satisfying the conditions~(\ref{eqn:mixed-re})
is necessary and standard in sparse recovery analysis.
This is closely related to the RIP condition of \cite{CandesTao05}
which assumes that there exist some $s>0$, and $\nu \in (0,1)$ such that
$\kappa(A,s) < (1+\nu)/(1 - \nu)$.
In fact, if RIP is satisfied with $\nu=0.2$ at $s=193(1+\gamma)\bar{s}$, 
then we may take $\gammaInc=2$ and $\tilde{s}=96(1+\gamma)\bar{s}$
so that the condition~(\ref{eqn:mixed-re}) is satisfied.
To see this, let $s=\bar{s}+2\tilde{s}$ and note that 
\[
\frac{1+\nu}{1-\nu} > \kappa(A,\bar{s}+2\tilde{s})
\geq
\frac{\rho_+(A, \bar{s}+2\tilde{s})}{\rho_-(A, \bar{s}+\tilde{s})} .
\]
Therefore we have
\[
\tilde{s} = 96 (1+\gamma)\bar{s} 
= 64 \frac{1+\nu}{1-\nu} (1+\gamma)\bar{s}
> 16 \frac{2 \rho_+(A, \bar{s}+2\tilde{s}) + 2\rho_+(A,\tilde{s})}
          {\rho_-(A, \bar{s}+\tilde{s})} (1+\gamma)\bar{s} .
\]
Although for practical purpose these constants are rather large, 
it is worth mentioning that our analysis focuses on the high level message,
without paying special attention to optimizing the constants.
\item
If $\Lmin >\gammaInc  \rho_+(A, \bar{s}+2\tilde{s})$,
then we may simply replace
$\gammaInc  \rho_+(A, \bar{s}+2\tilde{s})$
by $\Lmin$ in the assumption, and all theorem statements
hold with $\gammaInc  \rho_+(A, \bar{s}+2\tilde{s})$ replaced
by $\Lmin$.
Nevertheless in practice, it is natural to
simply pick 
\[
\Lmin=\rho_+(A,1)=\max_{i\in\{1,\ldots,n\}} \|A_i\|_2^2,
\]
where $A_i$ is the $i$-th column of~$A$.
It automatically satisfies the condition 
$\Lmin \leq \rho_+(A,\bar{s}+2\tilde{s})$.

\end{itemize}

Our first result below concerns the local geometric convergence of
Algorithm~\ref{alg:prox-grad}.
Basically, if the starting point $x^{(0)}$ is sparse
and the optimality condition is satisfied with adequate precision,
then all iterates along the solution path are sparse,
and Algorithm~\ref{alg:prox-grad} has geometric convergence.
To simplify the presentation, we use a single symbol~$\kappa$ to denote the 
restricted condition number
\[
\kappa = \kappa(A,\bar{s}+2\tilde{s}) 
= \frac{\rho_+(A,\bar{s}+2\tilde{s})}{\rho_-(A,\bar{s}+2\tilde{s})}.
\]

\begin{theorem}\label{thm:geometric-rate}
Suppose Assumption~\ref{asmp:mixed-re} holds.
If the initial point $x^{(0)}$ in Algorithm~\ref{alg:prox-grad} satisfies
\begin{equation}\label{eqn:x0-close}
\big\|x^{(0)}_{\bar{S}^c}\big\|_0 \leq \tilde{s} , \qquad
\omega_\lambda(x^{(0)})\leq \delta' \lambda,
\end{equation}
then for all $k\geq 0$, we have
\[
\big\|x^{(k)}_{\bar{S}^c}\big\|_0 \leq \tilde{s} , \qquad
\phi_\lambda(x^{(k)}) - \phi_\lambda^\star 
~\leq~ \left(1-\frac{1}{4\gammaInc\kappa} \right)^k 
\left(\phi_\lambda(x^{(0)})-\phi_\lambda^\star\right),
\] 
where 
$\phi_\lambda^\star = \phi_\lambda(x^\star(\lambda)) = \min_x \phi_\lambda(x)$.
\end{theorem}

Our next result gives the overall iteration complexity of the PGH method
in Algorithm~\ref{alg:homotopy}.
Roughly speaking, if the parameters~$\delta$ and~$\eta$ 
are chosen appropriately, then the total number of proximal-gradient steps 
for finding an $\epsilon$-optimal solution is $O(\ln(1/\epsilon))$.

\begin{theorem}\label{thm:overall-complexity}
Suppose Assumption \ref{asmp:mixed-re} holds with
$\lambda_\mathrm{tgt} \leq \lambda_0$
and the parameters~$\delta$ and~$\eta$ in Algorithm~\ref{alg:homotopy}
are chosen such that 
\[
\frac{1+\delta}{1+\delta'}\leq \eta < 1.
\]
Let $N = \left\lfloor\ln\left(\lambda_0/\lambda_\mathrm{tgt}\right) / 
         \ln\eta^{-1}\right\rfloor$
as in the algorithm. Then:
\begin{enumerate}
\item
The condition~(\ref{eqn:x0-close}) holds for each call of 
Algorithm~\ref{alg:prox-grad}. 
For $K=0,\ldots,N-1$, the number of proximal-gradient steps in each call of 
Algorithm~\ref{alg:prox-grad} is no more than
\[
\ln\left( \frac{C}{\delta^2} \right) \Bigg/
\ln\left( 1 - \frac{1}{4\gammaInc \kappa}\right)^{-1} ,
\]
where 
$C = 8 \gammaInc (1+\kappa)^2 (1+\gamma)\kappa \bar{s}$.
Note that this bound is independent of $\lambda_K$.
\item
For $K=0,\ldots,N-1$, the outer-loop iterates $\hat x^{(K)}$ satisfies
\begin{equation}\label{eqn:outer-geometric}
\phi_{\lambda_\mathrm{tgt}}(\hat x^{(K)}) - \phi_{\lambda_\mathrm{tgt}}^\star  
\leq \eta^{2(K+1)} \,
\frac{4.5(1+\gamma) \lambda_0^2 \bar{s}}{\rho_-(A,\bar{s}+\tilde{s})},
\end{equation}
and the following bound on sparse recovery performance holds
\[
\|\hat x^{(K)} - \bar{x}\|_2 \leq \eta^{K+1} \, 
\frac{2\lambda_0 \sqrt{\bar{s}}}{\rho_-(A,\bar{s}+\tilde{s})} .
\]
\item
When Algorithm~\ref{alg:homotopy} terminates, the total number of 
proximal-gradient steps is no more than
\[
\left(
\frac{\ln(\lambda_0/\lambda_\mathrm{tgt})}{\ln\eta^{-1}} \,
\ln\!\left( \frac{ C }{\delta^2}\right) 
+ \ln \max\left(1, \frac{ \lambda_\mathrm{tgt}^2 C  }{\epsilon^2}\right) 
\right) \Bigg/
\ln\left( 1 - \frac{1}{4\gammaInc \kappa}\right)^{-1} ,
\]
and the output $\hat{x}^\mathrm{(tgt)}$ satisfies
\[
\phi_{\lambda_\mathrm{tgt}}(\hat{x}^\mathrm{(tgt)})
- \phi_{\lambda_\mathrm{tgt}}^\star 
\leq \frac{4(1+\gamma) \lambda_\mathrm{tgt} \bar{s}}
          {\rho_-(A,\bar{s}+\tilde{s})} \,\epsilon.
\]
\end{enumerate}
\end{theorem}

We have the following remarks regarding these results:
\begin{itemize}

\item
The precision~$\epsilon$ in Algorithm~\ref{alg:homotopy} is measured 
against the optimality residue $\omega_\lambda(x)$.
In terms of the objective gap, suppose $\epsilon_0>0$ is the target precision 
to be reached.
Let
\[
K_0= \left\lceil \frac{1}{2} \ln \left(\frac{4.5 (1+\gamma) \lambda_{0}^2 \bar{s}}{\rho_-(A, \bar{s}+\tilde{s}) \epsilon_0} \right)\bigg/\ln \eta^{-1} \right\rceil - 1 .
\]
From the inequality~(\ref{eqn:outer-geometric}), 
we see that if $0 \leq K_0 \leq N-1$, then for all $K\geq K_0$,
\[
\phi_{\lambda_\mathrm{tgt}}(\hat x^{(K)}) - \phi_{\lambda_\mathrm{tgt}}^\star
\leq \epsilon_0 .
\]
If we let $\epsilon_0\to 0$ and run the PGH method forever,
then the number of proximal-gradient iterations is no more than 
$O(\ln(\lambda_0/\epsilon_0))$ to achieve an $\epsilon_0$ accuracy 
both on the gap of objective value and on the optimality residue
$\omega_\lambda(\cdot) \leq \epsilon_0$. 
This means that the PGH method achieves a global geometric rate of convergence.

\item
When the restricted condition number~$\kappa$ is large, we can use 
the approximation
\[
\ln\left( 1 - \frac{1}{4\gammaInc \kappa}\right)^{-1} 
\approx \frac{1}{4\gammaInc \kappa}.
\]
Then the overall iteration complexity can be estimated by
$O\left(\kappa\, \ln \left( \lambda_0/\epsilon\right) \right)$,
which is proportional to the restricted condition number $\kappa$.

\item
Even if we solve each stage to high precision with 
$\hat{\epsilon}_{K+1}= \min(\epsilon,\delta \lambda_{K+1})$,
the global convergence rate is still near geometric, and the total number 
of proximal-gradient steps is no more than $O( (\ln (\lambda_0/\epsilon))^2)$.
\end{itemize}

Theorem~\ref{thm:overall-complexity} plus restricted strong convexity 
immediately implies that the approximate solutions~$\hat x^{(K)}$ 
(and the last step solution $\hat{x}^\mathrm{(tgt)}$) also converge to 
$x^\star(\lambda_\mathrm{tgt})$ at a globally geometric rate.
A particularly interesting case is noise-free compressed sensing using 
the BP formulation~(\ref{eqn:BP}), which has the optimal solution $\bar{x}$.
For this problem, we can simply run  Algorithm~\ref{alg:homotopy} with 
$\lambda_\mathrm{tgt} = 0$ to solve~(\ref{eqn:BP}). 
While the convergence metrics such as objective value gap or optimality
residue are no longer informative in this case, 
Theorem~\ref{thm:overall-complexity} implies geometric convergence
of the recovery error $\|\hat x^{(K)} - \bar{x}\|_2$.
More precisely, we have:

\begin{corollary}
Suppose $b=A\bar{x}$ and the assumptions stated in 
Theorem~\ref{thm:overall-complexity} hold.
We can choose an arbitrarily small $\lambda_\mathrm{tgt}>0$ 
in Algorithm~\ref{alg:homotopy}, and after $K$ outer iterations, we have
\[
\|\hat x^{(K)} - \bar{x}\|_2 \leq \eta^{K+1} 
\frac{2\lambda_0 \sqrt{\bar{s}}}{\rho_-(A,\bar{s}+\tilde{s}) }.
\]
\end{corollary}

Note that part~1 of Theorem~\ref{thm:overall-complexity} implies that
$K$ outer iterations of Algorithm~\ref{alg:homotopy} requires 
no more than $O(K)$ proximal-gradient steps.
This result can be interpreted as a global geometric rate of convergence 
for solving the BP problem.

\section{Proofs of convergence results}
\label{sec:analysis}
The proofs of our convergence results are divided into 
the following subsections.
In Section~\ref{sec:sparse-path}, we show that under
Assumption~\ref{asmp:mixed-re}, 
if $x^{(0)}$ is sparse and $\omega_\lambda(x^{(0)})$ is small,
then all iterates generated by Algorithm~\ref{alg:prox-grad} are sparse.
In Section~\ref{sec:geometric-rate}, we use the sparsity along the solution 
path and the restricted eigenvalue condition to show the local geometric 
convergence of Algorithm~\ref{alg:prox-grad}, thus proving
Theorem~\ref{thm:geometric-rate}.
In Section~\ref{sec:overall-complexity}, we show that by setting the parameters
$\delta$ and $\eta$ in Algorithm~\ref{alg:homotopy} appropriately, 
we have geometric convergence at each stage of the homotopy method, 
which leads to the global iteration complexity $O(\log(1/\epsilon))$.

\subsection{Sparsity along the solution path}
\label{sec:sparse-path}
First, we list some useful inequalities that are direct consequences of the
assumption~(\ref{eqn:big-lambda}):
\begin{eqnarray}
(1-\delta')\lambda - \|A^T z\|_\infty &>& 0  
    \label{eqn:big-lambda-1}\\
(1+\delta')\lambda + \|A^T z\|_\infty &\leq& 2 \lambda 
    \label{eqn:big-lambda-2}\\
\lambda + \|A^T z\|_\infty &\leq& (2-\delta')\lambda 
    \label{eqn:big-lambda-3}\\ 
\frac{(1+\delta')\lambda+\|A^T z\|_\infty}{(1-\delta')\lambda-\|A^T z\|_\infty}
&\leq& \gamma .
    \label{eqn:big-lambda-4}
\end{eqnarray}

The following result means that if $x$ is sparse, and it satisfies an 
approximate optimality condition for minimizing~$\phi_\lambda$, 
then $\phi_\lambda(x)$ is not much larger than $\phi_\lambda(\bar{x})$.

\begin{lemma}\label{lem:kkt-consequences}
Suppose Assumption~\ref{asmp:mixed-re} holds, 
and  $\lambda \geq \lambda_\mathrm{tgt}$.  
If $x$ is sparse, i.e., $\|x_{\bar{S}^c}\|_0 \leq \tilde{s}$, 
and it satisfies the approximate optimality condition
\begin{equation}\label{eqn:approx-kkt}
\min_{\xi\in\partial\|x\|_1} \left\| A^T(A x-b) + \lambda \xi \right\|_\infty 
\leq \delta' \lambda ,
\end{equation}
then we have
\begin{equation}\label{eqn:l1-bound}
\| (x-\bar{x})_{\bar{S}^c}\|_1\leq \gamma   \| (x-\bar{x})_{\bar{S}}\|_1 
\end{equation}
and  
\begin{equation}\label{eqn:l2-bound}
\|x-\bar{x}\|_2\leq\frac{2 \lambda\sqrt{\bar{s}} }{\rho_-(A,\bar{s}+\tilde{s})}
\end{equation}
and
\begin{equation}\label{eqn:obj-bound}
\phi_\lambda(x) \leq \phi_\lambda(\bar{x}) + 
\frac{2 \delta' (1+\gamma) \lambda^2 \bar{s} }{ \rho_-(A,\bar{s}+\tilde{s}) }.
\end{equation}
\end{lemma}

\begin{proof}
Let $\xi\in\partial\|x\|_1$ be a subgradient that achieves the minimum on 
the left-hand side of~(\ref{eqn:approx-kkt}). 
Then the approximate optimality condition leads to
\begin{eqnarray*}
(x-\bar{x})^T \left( A^T(A x-b) + \lambda \xi \right)
&\leq& \|x-\bar{x}\|_1 \left\|A^T(A x-b) + \lambda \xi \right\|_\infty \\
&\leq& \delta'\lambda \|x-\bar{x}\|_1 .
\end{eqnarray*}
On the other hand, we can use $b=A\bar{x}+z$ to obtain
\begin{eqnarray*}
(x-\bar{x})^T \left( A^T(A x-b) + \lambda \xi \right)
&=& (x-\bar{x})^T A^T \bigl(A(x-\bar{x})-z\bigr) + \lambda(x-\bar{x})^T\xi \\
&=& \left\| A(x-\bar{x}) \right\|_2^2 - (x-\bar{x})^T A^T z 
  + \lambda\, \xi^T (x-\bar{x}) \\
&\geq& \left\| A(x-\bar{x}) \right\|_2^2 - \|x-\bar{x}\|_1 \|A^T z\|_\infty
  + \lambda \, \xi^T (x-\bar{x}).
\end{eqnarray*}
Next, we break the inner product $\xi^T(x-\bar{x})$ into two parts as
\[
\xi^T(x-\bar{x}) = \xi_{\bar{S}}^T (x-\bar{x})_{\bar{S}}
+ \xi_{\bar{S}^c}^T (x-\bar{x})_{\bar{S}^c}.
\]
For the first part, we have (by noticing $\|\xi\|_\infty\leq 1$)
\[
\xi_{\bar{S}}^T (x-\bar{x})_{\bar{S}} ~\geq~
-\|\xi_{\bar{S}}\|_\infty \|(x-\bar{x})_{\bar{S}}\|_1 
~\geq~ - \|(x-\bar{x})_{\bar{S}}\|_1 .
\]
For the second part, we use the facts $\bar{x}_{\bar{S}^c}=0$ and
$\xi \in \partial\|x\|_1$ to obtain
\[
\xi_{\bar{S}^c}^T (x-\bar{x})_{\bar{S}^c} ~=~ x_{\bar{S}^c}^T\xi_{\bar{S}^c} 
~=~ \|x_{\bar{S}^c}\|_1 ~=~  \|(x-\bar{x})_{\bar{S}^c}\|_1 .
\]
Combining the inequalities above gives
\[
\left\| A(x-\bar{x}) \right\|_2^2 - \|A^T z\|_\infty \|x-\bar{x}\|_1 
- \lambda \|(x-\bar{x})_{\bar{S}}\|_1 + \lambda \|(x-\bar{x})_{\bar{S}^c}\|_1 
~\leq~ \delta'\lambda \|x-\bar{x}\|_1 .
\]
Using
$\|x-\bar{x}\|_1 = \|(x-\bar{x})_{\bar{S}}\|_1 + \|(x-\bar{x})_{\bar{S}^c}\|_1$
and rearranging terms, we arrive at
\begin{equation}\label{eqn:l1l2-bound}
\left\| A(x-\bar{x}) \right\|_2^2 +\left((1\!-\!\delta')\lambda 
-\|A^T z\|_\infty\right) \|(x-\bar{x})_{\bar{S}^c}\|_1
~\leq~ 
\left((1\!+\!\delta')\lambda +\|A^T z\|_\infty\right) 
\|(x-\bar{x})_{\bar{S}}\|_1 .
\end{equation}
By further using the inequalities~(\ref{eqn:big-lambda-1})
and~(\ref{eqn:big-lambda-4}), we obtain
\[
\| (x-\bar{x})_{\bar{S}^c}\|_1\leq \gamma   \| (x-\bar{x})_{\bar{S}}\|_1 ,
\]
which is the first desired result in~(\ref{eqn:l1-bound}).

Since by assumption $\|x_{\bar{S}^c}\|_0\leq \tilde{s}$,
we can use the restricted eigenvalue condition to obtain
\begin{eqnarray*}
\rho_-(A,\bar{s}+\tilde{s}) \|x-\bar{x}\|_2^2 
&\leq & \|A(x-\bar{x})\|_2^2 \\
&\leq& \left((1+\delta')\lambda +\|A^T z\|_\infty\right) 
       \|(x-\bar{x})_{\bar{S}}\|_1 \\
&\leq& 2 \lambda \|(x-\bar{x})_{\bar{S}}\|_1 \\
&\leq& 2 \lambda \sqrt{\bar{s}} \, \|(x-\bar{x})_{\bar{S}}\|_2 \\
&\leq& 2 \lambda \sqrt{\bar{s}} \, \|x-\bar{x}\|_2 ,
\end{eqnarray*}
where the second inequality is a result of~(\ref{eqn:l1l2-bound}), 
the third inequality follows from~(\ref{eqn:big-lambda-2}),
and the fourth inequality holds because $|\bar{S}|=\bar{s}$.
This proves the second desired bound in~(\ref{eqn:l2-bound}).

Finally, since $\phi_\lambda$ is convex and 
$A^T(A x-b)+\xi$ is a subgradient of~$\phi$ at~$x$, we have
\[
\phi_\lambda(x) - \phi_\lambda(\bar{x}) 
~\leq~ - \left(A^T (A x-b) + \xi \right)^T (\bar{x}-x)
~\leq~ \delta' \lambda \|\bar{x}-x\|_1 .
\]
From the inequality in~(\ref{eqn:l1-bound}), we have
\[
\|\bar{x}-x\|_1 
~=~ \|(\bar{x}-x)_{\bar{S}}\|_1 + \|(\bar{x}-x)_{\bar{S}^c}\|_1
~\leq~ (1+\gamma) \|(\bar{x}-x)_{\bar{S}}\|_1.
\]
Therefore,
\[
\phi_\lambda(x) - \phi_\lambda(\bar{x}) 
~\leq~ \delta' \lambda (1+\gamma) \|(\bar{x}-x)_{\bar{S}}\|_1
~\leq~ \delta' \lambda (1+\gamma) \sqrt{\bar{s}}\, \|(\bar{x}-x)_{\bar{S}}\|_2,
\]
which, together with~(\ref{eqn:l2-bound}), leads to the third desired result. 
\end{proof}

The following result means that if $x$ is sparse, and $\phi_\lambda(x)$ 
is not much larger than $\phi_\lambda(\bar{x})$, 
then both $\|x-\bar{x}\|_2$ and $\|x-\bar{x}\|_1$ are small.

\begin{lemma}\label{lem:bound-norms}
Suppose Assumption~\ref{asmp:mixed-re} holds, 
and $\lambda \geq \lambda_\mathrm{tgt}$.
Consider $x$ such that
\[
\|x_{\bar{S}^c}\|_0 \leq \tilde{s} , 
\qquad \phi_\lambda(x) \leq \phi_\lambda(\bar{x}) 
+ \frac{2 \delta' (1+\gamma) \lambda^2 \bar{s}}{\rho_-(A,\bar{s}+\tilde{s}) } ,
\]
then 
\[
\max\left\{ \frac{1}{2\lambda} \|A (x-\bar{x})\|_2^2,~\|x-\bar{x}\|_1 \right\}
\leq \frac{4 (1+\gamma) \lambda \bar{s} }{\rho_-(A,\bar{s}+\tilde{s})} .
\]
\end{lemma}
In fact, similar results holds under the condition 
$\omega_\lambda(x)\leq \delta'\lambda$, and are already proved in
Lemma~\ref{lem:kkt-consequences}.
However, in the proximal gradient method, the optimality residue
$\omega_\lambda(x^{(k)})$ may not be monotonic decreasing, 
but the objective function $\phi_\lambda(x^{(k)})$ is.
So in order to establish the desired results for all iterates along 
the solution path, we need to show them when the objective function is
sufficiently small, which is more involved.

\begin{proof}
For notational convenience, let
\[
\Delta = 
\frac{2 \delta' (1+\gamma) \lambda^2 \bar{s}}{\rho_-(A,\bar{s}+\tilde{s}) } .
\]
We write the assumption $\phi_\lambda(x)\leq\phi_\lambda(\bar x)+\Delta$ 
explicitly as
\begin{equation}\label{eqn:close-obj}
\frac{1}{2}\|Ax-b\|_2^2 + \lambda \|x\|_1 \leq
\frac{1}{2}\|A\bar x-b\|_2^2 + \lambda \|\bar x\|_1 + \Delta.
\end{equation}
We can expand the least-squares part in $\phi_\lambda(x)$ as
\begin{eqnarray*}
\frac{1}{2}\|Ax-b\|_2^2  
&=& \frac{1}{2}\|(A\bar x-b) + A(x-\bar x)\|_2^2  \\
&=& \frac{1}{2}\|(A\bar x-b)\|_2^2 + \frac{1}{2}\|A(x-\bar x)\|_2^2 
    +(x-\bar x)^T A^T (A\bar x-b)\\
&\geq& \frac{1}{2}\|(A\bar x-b)\|_2^2 + \frac{1}{2}\|A(x-\bar x)\|_2^2
    -\|x-\bar x\|_1 \| A^T (A\bar{x}-b)\|_\infty .
\end{eqnarray*}
Plugging the above inequality into~(\ref{eqn:close-obj}),
and noticing $A\bar{x}-b=z$, we obtain
\[
\frac{1}{2}\|A(x-\bar x)\|_2^2 - \|x-\bar x\|_1 \| A^T z\|_\infty  
+ \lambda \|x\|_1 ~\leq~ \lambda \|\bar x\|_1 + \Delta.
\]
Using the fact $\bar x_{\bar{S}^c}=0$, we have
\[
\|x\|_1 
= \|x_{\bar{S}^c}\|_1 + \|x_{\bar{S}}\|_1
= \|x_{\bar{S}^c} - \bar x_{\bar{S}^c}\|_1 + \|x_{\bar{S}}\|_1 .
\]
Therefore
\begin{eqnarray*}
 \frac{1}{2}\|A(x-\bar x)\|_2^2 - \|x-\bar x\|_1 \| A^T z\|_\infty  
       + \lambda \|x_{\bar{S}^c} - \bar x_{\bar{S}^c}\|_1 
&\leq& \lambda \left( \|\bar x_{\bar{S}}\|_1 - \|x_{\bar{S}}\|_1 \right) 
       + \Delta \\
&\leq& \lambda \, \|\bar x_{\bar{S}} - x_{\bar{S}}\|_1 + \Delta.
\end{eqnarray*}
By further splitting $\|x-\bar{x}\|_1$ on the left-hand side as 
$\|(x-\bar{x})_{\bar{S}}\|_1 + \|(x-\bar{x})_{\bar{S}^c}\|_1$, we get
\begin{equation}\label{eqn:two-cases}
\frac{1}{2} \|A (x-\bar{x})\|_2^2 + 
 \left(\lambda -\|A^T z\|_\infty\right) \|(x-\bar{x})_{\bar{S}^c}\|_1
~\leq~ \left(\lambda +\|A^T z\|_\infty \right) \|(x-\bar{x})_{\bar{S}}\|_1 
+\Delta.
\end{equation}

Now there are two possible cases. 
In the first case, we assume
\begin{equation}\label{eqn:lem-err-1}
\|x-\bar{x}\|_1
~\leq~ \frac{\Delta}{\delta'\lambda}
~=~ \frac{2 (1+\gamma) \lambda \bar{s}}{ \rho_-(A,\bar{s}+\tilde{s})} .
\end{equation}
From~(\ref{eqn:big-lambda-1}), we know that
$\left(\lambda -\|A^T z\|_\infty\right) \|(x-\bar{x})_{\bar{S}^c}\|_1$
is nonnegative, so we can drop it from the left-hand side 
of~(\ref{eqn:two-cases}) to obtain
\begin{eqnarray*}
\frac{1}{2} \|A (x-\bar{x})\|_2^2  
&\leq& \left(\lambda +\|A^T z\|_\infty \right) \|(x-\bar{x})_{\bar{S}}\|_1 
  + \Delta\\
&\leq& (2\lambda -\delta' \lambda) \|(x-\bar{x})_{\bar{S}}\|_1  + \Delta\\
&\leq& (2\lambda -\delta' \lambda) \frac{2 (1+\gamma) \lambda \bar{s} }{ 
  \rho_-(A,\bar{s}+\tilde{s})}
 + \frac{2 \delta' (1+\gamma) \lambda^2 \bar{s}}{\rho_-(A,\bar{s}+\tilde{s})}\\
&=& \frac{4\lambda (1+\gamma) \lambda \bar{s} }{\rho_-(A,\bar{s}+\tilde{s})},
\end{eqnarray*}
where in the second inequality we used~(\ref{eqn:big-lambda-3}), 
and in the third inequality we used~(\ref{eqn:lem-err-1}).
This means the claim holds.

In the second case, the assumption in~(\ref{eqn:lem-err-1}) does not hold.
Then $\Delta<\delta'\lambda\|x-\bar{x}\|_1$ 
and~(\ref{eqn:two-cases}) implies
\[
\frac{1}{2} \|A (x-\bar{x})\|_2^2 + 
\left(\lambda -\|A^T z\|_\infty \right) \|(x-\bar{x})_{\bar{S}^c}\|_1
\leq \left(\lambda +\|A^T z\|_\infty \right) \|(x-\bar{x})_{\bar{S}}\|_1 
  + \delta' \lambda   \|x-\bar{x}\|_1 .
\]
Again we split $\|x-\bar{x}\|_1$ as 
$\|(x-\bar{x})_{\bar{S}}\|_1 + \|(x-\bar{x})_{\bar{S}^c}\|_1$ to obtain
\begin{equation}\label{eqn:pre-rec}
\frac{1}{2} \|A (x-\bar{x})\|_2^2 + 
\left((1-\delta')\lambda -\|A^T z\|_\infty\right) \|(x-\bar{x})_{\bar{S}^c}\|_1
~\leq~ 
\left((1+\delta')\lambda +\|A^T z\|_\infty\right) \|(x-\bar{x})_{\bar{S}}\|_1 .
\end{equation}
By further using the inequalities~(\ref{eqn:big-lambda-1})
and~(\ref{eqn:big-lambda-4}), we get
\begin{equation}\label{eqn:l1-bound-1}
\|(x-\bar{x})_{\bar{S}^c}\|_1 ~\leq~
\frac{(1+\delta')\lambda+\|A^T z\|_\infty}{(1-\delta')\lambda-\|A^T z\|_\infty}
\|(x-\bar{x})_{\bar{S}}\|_1
~\leq~ \gamma \|(x-\bar{x})_{\bar{S}}\|_1.
\end{equation}
Moreover, we can use the restricted eigenvalue condition and the assumption
$\|x_{\bar{S}^C}\|_0\leq\tilde{s}$ to obtain
\begin{eqnarray*}
\frac{1}{2}\rho_-(A,\bar{s}+\tilde{s}) \|x-\bar{x}\|_2^2 
&\leq& \frac{1}{2} \|A (x-\bar{x})\|_2^2 \\
&\leq& \left((1+\delta')\lambda +\|A^T z\|_\infty\right) 
       \|(x-\bar{x})_{\bar{S}}\|_1  \\
&\leq& 2 \lambda \|(x-\bar{x})_{\bar{S}}\|_1  \\
&\leq& 2 \lambda \sqrt{\bar{s}} \, \|(x-\bar{x})_{\bar{S}}\|_2  \\
&\leq& 2 \lambda \sqrt{\bar{s}} \, \|x-\bar{x}\|_2  ,
\end{eqnarray*}
where the second inequality follows from~(\ref{eqn:pre-rec}),
the third inequality follows from~(\ref{eqn:big-lambda-2}),
and the forth inequality holds because $|\bar{S}|=\bar{s}$.
Hence
\[
\|x-\bar{x}\|_2 
\leq \frac{4 \lambda \sqrt{\bar{s}}}{\rho_-(A,\bar{s}+\tilde{s})}.
\]
The above arguments also imply
\[
\frac{1}{2} \|A (x-\bar{x})\|_2^2 
~\leq~ 2 \lambda \sqrt{\bar{s}} \, \|x-\bar{x}\|_2 
~\leq~ \frac{8 \lambda^2 \bar{s}}{\rho_-(A,\bar{s}+\tilde{s})}
~\leq~ \frac{4(1+\gamma) \lambda^2 \bar{s}}{\rho_-(A,\bar{s}+\tilde{s})},
\]
where the last inequality holds because $\gamma>1$.
Finally, using~(\ref{eqn:l1-bound-1}), we get
\[
\|x-\bar{x}\|_1 ~\leq~ (1+\gamma)\|(x-\bar{x})_{\bar{S}}\|_1
~\leq~ (1+\gamma)\sqrt{\bar{s}} \, \|(x-\bar{x})_{\bar{S}}\|_2
~\leq~ \frac{4 (1+\gamma) \lambda \bar{s}}{\rho_-(A,\bar{s}+\tilde{s})}.
\]
These prove the desired bound. 
\end{proof}

The following lemma means that if $x$ is sparse and $\phi_\lambda(x)$ 
is not much larger than $\phi_\lambda(\bar{x})$, 
then $T_{\lambda,L}(x)$ is sparse.

\begin{lemma}\label{lem:local-sparse}
Suppose Assumption~\ref{asmp:mixed-re} holds, 
and $\lambda \geq \lambda_\mathrm{tgt}$. 
Suppose $x$ satisfies
\begin{equation}\label{eqn:obj-close}
\|x_{\bar{S}^c}\|_0 \leq \tilde{s} , \qquad 
\phi_\lambda(x) \leq \phi_\lambda(\bar{x}) 
+ \frac{2 \delta' (1+\gamma) \lambda^2 \bar{s} }{\rho_-(A,\bar{s}+\tilde{s})} ,
\end{equation}
and $L<\gammaInc\rho_+(A,\bar{s}+2\tilde{s})$.
Then 
\[
\left\|\bigl(T_{\lambda,L}(x)\bigr)_{\bar{S}^c}\right\|_0 < \tilde{s}.
\]
\end{lemma}

\begin{proof}
Recall that $T_{\lambda,L}$ can be computed by the soft-thresholding operator
as in~(\ref{eqn:IST}).
That is,
\[
(T_L(x))_i = \sgn(\tilde x_i) \max \left\{ |\tilde x_i|
 - \frac{\lambda}{L}, ~0 \right\}, \qquad i=1,\ldots,n,
\]
where
\[
\tilde x = x - \frac{1}{L} A^T( Ax-b) 
= x - \frac{1}{L} A^T A(x-\bar x) + \frac{1}{L}A^T z.
\]
In order to upper bound the number of nonzero elements 
in $(T_L(x))_{\bar{S}^c}$,
we split the truncation threshold $\lambda/L$ on elements of
$\tilde{x}_{\bar{S}^c}$ into three parts: 
\begin{itemize}
\item $\lambda/4L$ on elements of $x_{\bar{S}^c}$,
\item $\lambda/4L$ on elements of $(1/L)A^T z$, and
\item $\lambda/2L$ on elements of $(1/L)A^T A(x-\bar{x})$.
\end{itemize}
Since by assumption $\|A^T z\|_\infty \leq \lambda/4$, we have
$\bigl|\{j:((1/L)A^T z)_j > \lambda/4L\}\bigr|=0$.
Therefore,
\[
\left\|\bigl(T_L(x)\bigr)_{\bar{S}^c} \right\|_0
~\leq~ \left|\bigl\{j\in\bar{S}^c : |x_j| > \lambda/4L\bigr\} \right|
+ \bigl| \bigl\{j: \bigl|\bigl(A^T A (x-\bar{x})\bigr)_j \bigr| \geq \lambda/2 
  \bigr\} \bigr| .
\]
Note that 
\begin{eqnarray}
\bigl|\{j\in\bar{S}^c :|x_j| \geq \lambda/4L\} \bigr|
&=& \bigl|\{j\in\bar{S}^c :|(x-\bar{x})_j| \geq \lambda/4L\} \bigr| \nonumber\\
&\leq& \bigl|\{j :|(x-\bar{x})_j| \geq \lambda/4L\} \bigr| \nonumber\\
&\leq& 4L \lambda^{-1} \|x-\bar{x}\|_1 \nonumber\\
&\leq& \frac{16 L (1+\gamma) \bar{s} }{\rho_-(A,\bar{s}+\tilde{s})} ,
\label{eqn:big-x-bound}
\end{eqnarray}
where the last inequality follows from Lemma~\ref{lem:bound-norms}.

For the last part, consider $S'$ with maximum size $s'=|S'| \leq \tilde{s}$ such that
\[
S' \subset \{j: |(A^T A (x-\bar{x}))_j| \geq \lambda /2 \} .
\]
Then there exists $u$ such that $\|u\|_\infty=1$ and $\|u\|_0=s'$, and
$s' \lambda/2 \leq  u^T A^T A (x-\bar{x})$.
Moreover,
\[
s' \lambda/2 ~\leq~  u^T A^T A (x-\bar{x})
~\leq~ \|A u\|_2 \|A (x-\bar{x})\|_2 
~\leq~ \sqrt{\rho_+(A,s')} \sqrt{s'} 
\sqrt{\frac{8(1+\gamma)\lambda^2 \bar{s}}{\rho_-(A,\bar{s}+\tilde{s})} } ,
\]
where the last inequality again follows from Lemma~\ref{lem:bound-norms}.
Taking squares of both sides of the above inequality gives
\[
s'\leq \frac{32\,\rho_+(A,s') (1+\gamma) \bar{s} }{\rho_-(A,\bar{s}+\tilde{s})} 
\leq 
\frac{32\,\rho_+(A,\tilde{s}) (1+\gamma) \bar{s} }{\rho_-(A,\bar{s}+\tilde{s})} 
< \tilde{s} ,
\]
where the last inequality is due to (\ref{eqn:mixed-re}).
Since $s'=|S'|$ achieves the maximum possible value such that 
$s' \leq \tilde{s}$
for any subset $S'$ of $\{j: |(A^T A (x^{(k)}-\bar{x}))_j| \geq \lambda /2 \}$,
and the above inequality shows that $s' < \tilde{s}$, we must have
\[
S' = \{j: |(A^T A (x^{(k)}-\bar{x}))_j| \geq \lambda /2 \} ,
\]
and thus
\[
\bigl| \{j: |(A^T A (x^{(k)}-\bar{x}))_j| \geq \lambda /2 \} \bigr|  =s' \leq 
\left\lfloor \frac{32\,\rho_+(A,\tilde{s}) (1+\gamma) \bar{s} }{\rho_-(A,\bar{s}+\tilde{s})} \right\rfloor .
\]
Finally, combining the above bound  with the bound 
in~(\ref{eqn:big-x-bound}) gives 
\[
\left\|\bigl(T_{\lambda,L}(x)\bigr)_{\bar{S}^c}\right\|_0 ~\leq~
\frac{16 \left(L + 2\rho_+(A, \tilde{s}) \right)}{\rho_-(A, \bar{s}+\tilde{s})}
(1+\gamma)\bar{s} .
\]
Under the assumption $L<\gammaInc\rho_+(A,\bar{s}+2\tilde{s})$ 
and~(\ref{eqn:mixed-re}),
the right-hand side of the above inequality is less than~$\tilde{s}$.
This proves the desired result.
\end{proof}

Recall that each iteration of Algorithm~\ref{alg:prox-grad} takes the form
$x^{(k+1)}=T_{\lambda,M_k}(x^{(k)})$.
According to~(\ref{eqn:monotone-decrease}), the objective value 
$\phi_\lambda(x^{(k)})$ is monotone decreasing.
So if $x^{(0)}$ satisfies the condition~(\ref{eqn:obj-close}),
every iterate $x^{(k)}$ satisfies the same condition.
In order to show 
\[
\|(x^{(k)})_{\bar{S}^c}\|_0 < \tilde{s}, \quad \forall\, k>0,
\]
we only need to note that the line-search procedure 
(Algorithm~\ref{alg:line-search}) always terminates with 
\begin{equation}\label{eqn:local-lipschitz}
M_k\leq\gammaInc\rho_+(A,\bar{s}+2\tilde{s}).
\end{equation}
Indeed, as long as
\[
M_k \in [\rho_+(A,\bar{s}+2\tilde{s}), \gammaInc\rho_+(A,\bar{s}+2\tilde{s}) ], 
\]
Lemma~\ref{lem:local-sparse} implies that
$\left\|\bigl(T_{\lambda,L}(x)\bigr)_{\bar{S}^c}\right\|_0 < \tilde{s}$
and the restricted smoothness property~(\ref{eqn:restr-smoothness}) implies the termination of line-search.

\subsection{Proof of Theorem~\ref{thm:geometric-rate}}
\label{sec:geometric-rate}
In this subsection, we show that for any fixed~$\lambda$, the sequence
$\bigl\{x^{(k)}\bigr\}_{k=0}^\infty$ generated by Algorithm~\ref{alg:prox-grad}
(without invoking the stopping criteria) has a limit 
and the local rate of convergence is geometric. 

First, since the sub-level set 
$\{x:\phi_\lambda(x)\leq\phi_\lambda(x^{(0)})\}$ is 
bounded and $\phi_\lambda(x^{(k)})$ is monotone decreasing, 
the sequence $\bigl\{x^{(k)}\bigr\}_{k=0}^\infty$ is bounded.
By the Bolzano-Weierstrass theorem, it has a convergent subsequence and
a corresponding accumulation point.
Moreover, from the inequality~(\ref{eqn:monotone-decrease}) and the fact
that $\phi_\lambda(x)$ is bounded below, we conclude that 
\[
\lim_{k\to\infty} \|g_{\lambda,L}(x^{(k)})\|_2 = 0.
\]
By Lemma~\ref{lem:kkt-grad-mapping}, this implies that any accumulation point 
of the sequence $\bigl\{x^{(k)}\bigr\}_{k=0}^\infty$
satisfies the optimality condition, therefore is a minimizer of $\phi_\lambda$.

Let $x^\star(\lambda)$ denote an accumulation point of the sequence
$\bigl\{x^{(k)}\bigr\}_{k=0}^\infty$.
As a consequence of Lemma~\ref{lem:local-sparse}, any accumulation point 
is also sparse; In particular, we have 
$\|(x^\star(\lambda))_{\bar{S}^c}\|_0 \leq \tilde{s}$.

Now using the restricted strong convexity 
property~(\ref{eqn:restr-strong-convex}), we have
\begin{equation}\label{eqn:f-strong-convex}
f(x) \geq f(x^\star) + \langle \nabla\!f(x^\star(\lambda)),x-x^\star(\lambda)
\rangle + \frac{\rho_-(A,\bar{s}+2\tilde{s})}{2} \|x-x^\star(\lambda)\|_2^2.
\end{equation}
Since $x^\star(\lambda) = \argmin_x \{f(x) + \lambda \|x\|_1\}$,
there must exists $\xi\in\partial\|x^\star(\lambda)\|_1$ such that
\begin{equation}\label{eqn:kkt-condition}
    \nabla\!f(x^\star(\lambda)) + \lambda \xi = 0.
\end{equation}
Since $\xi\in\partial\|x^\star(\lambda)\|_1$, we also have
\begin{equation}\label{eqn:l1-convex}
\lambda\|x\|_1 \geq \lambda\|x^\star(\lambda)\|_1 
+ \langle \lambda\xi, x-x^\star(\lambda) \rangle.
\end{equation}
Adding the two inequalities~(\ref{eqn:f-strong-convex}) 
and~(\ref{eqn:l1-convex}) and using~(\ref{eqn:kkt-condition}), we get
\begin{equation}\label{eqn:strong-convex}
\phi_\lambda(x) -\phi_\lambda(x^\star(\lambda)) 
\geq \frac{\rho_-(A,\bar{s}+2\tilde{s})}{2} \|x-x^\star(\lambda)\|_2^2,
\qquad \forall\, x:\|x_{\bar{S}^c}\|_0 \leq \tilde{s}.
\end{equation}

Since any accumulation point satisfies $\|x_{\bar{S}^c}\|_0 \leq \tilde{s}$,
we conclude that $x^\star(\lambda)$ is a unique accumulation point,
in other words, the limit, of the sequence
$\bigl\{x^{(k)}\bigr\}_{k=0}^\infty$.

Next we show that under the assumptions in Lemma~\ref{lem:local-sparse},
especially with $x^{(0)}$ satisfying~(\ref{eqn:obj-close}),
Algorithm~\ref{alg:prox-grad} has a geometric convergence rate.
We start with the stopping criteria in the line search procedure:
\begin{eqnarray*}
\phi_\lambda(x^{(k+1)})
&\leq& \psi_{\lambda,M_k}(x^{(k)}, x^{(k+1)}) \\
&\leq& \min_x \left\{ f(x) + \frac{M_k}{2} \|x-x^{(k)}\|_2^2
       + \lambda \|x\|_1 \right\} \\
&=& \min_x \left\{ \phi_\lambda(x) + \frac{M_k}{2} \|x-x^{(k)}\|_2^2 \right\}.
\end{eqnarray*}
where the second inequality follows from the convexity of~$f$.
We can further relax the right-hand side of the above inequality  
by restricting the minimization over the line segment
$x = \alpha x^\star(\lambda) + (1-\alpha)x^{(k)}$, where $\alpha\in[0,1]$. 
This leads to
\begin{eqnarray*}
\phi_\lambda(x^{(k+1)})
&\leq& \min_\alpha \left\{ \phi_\lambda\bigl(\alpha x^\star(\lambda)
       +(1-\alpha)x^{(k)} \bigr) 
       + \frac{M_k}{2} \|\alpha(x^{(k)}-x^\star(\lambda))\|_2^2 \right\} \\
&\leq& \min_\alpha \left\{ \alpha\phi_\lambda(x^\star(\lambda))
       +(1-\alpha)\phi_\lambda(x^{(k)}) 
       + \frac{\alpha^2 M_k}{2} \|x^{(k)}-x^\star(\lambda)\|_2^2 \right\} \\
&=& \min_\alpha \left\{ \phi_\lambda(x^{(k)}) 
       - \alpha \bigl( \phi_\lambda(x^{(k)})
       -\phi_\lambda(x^\star(\lambda)) \bigr)
       + \frac{\alpha^2 M_k}{2} \|x^{(k)}-x^\star(\lambda)\|_2^2 \right\} 
\end{eqnarray*}
Since the conclusion of Lemma~\ref{lem:local-sparse} implies that
$\|x^{(k)}_{\bar{S}^c}\|_0\leq \tilde{s}$ for all $k\geq 0$,
we can use the ``restricted'' strong convexity 
property~(\ref{eqn:strong-convex}) to obtain
\[
\phi_\lambda(x^{(k+1)}) \leq \min_\alpha \left\{ \phi_\lambda(x^{(k)}) 
- \alpha \left(1-\frac{\alpha M_k}{\rho_-(A,\bar{s}+2\tilde{s})} \right) 
\left(\phi_\lambda(x^{(k)})-\phi_\lambda(x^\star(\lambda))\right)\right\}.
\]
The minimizing value is $\alpha=\rho_-(A,\bar{s}+2\tilde{s})/(2M_k)$, 
which gives
\[
\phi_\lambda(x^{(k+1)}) ~\leq~ \phi_\lambda(x^{(k)}) 
- \frac{\rho_-(A,\bar{s}+2\tilde{s})}{4M_k}
\left(\phi_\lambda(x^{(k)})-\phi_\lambda(x^\star(\lambda))\right).
\]
Let $\phi_\lambda^\star = \phi_\lambda(x^\star(\lambda))$.
Subtracting $\phi_\lambda^\star$ from both side of the above inequality gives
\begin{eqnarray*}
\phi_\lambda(x^{(k+1)}) - \phi_\lambda^\star 
&\leq& \left( 1 - \frac{\rho_-(A,\bar{s}+2\tilde{s})}{4 M_k} \right)
    \left(\phi_\lambda(x^{(k)})-\phi_\lambda^\star\right) \\
&\leq& \left( 1 - \frac{\rho_-(A,\bar{s}+2\tilde{s})}{
    4 \gammaInc \rho_+(A,\bar{s}+2\tilde{s})} \right)
    \left(\phi_\lambda(x^{(k)})-\phi_\lambda^\star\right),
\end{eqnarray*}
where the second inequality follows from~(\ref{eqn:local-lipschitz}).
Therefore, we have
\[
\phi_\lambda(x^{(k)}) - \phi_\lambda^\star 
~\leq~ \left(1-\frac{1}{4\gammaInc\kappa} \right)^k 
\left(\phi_\lambda(x^{(0)})-\phi_\lambda^\star\right),
\] 
where
\[
\kappa=\frac{\rho_+(A,\bar{s}+2\tilde{s})}{\rho_-(A,\bar{s}+2\tilde{s})}
\]
is a restricted condition number.
Note that the above convergence rate does not depend on~$\lambda$.

\subsection{Proof of Theorem~\ref{thm:overall-complexity}}
\label{sec:overall-complexity}

In Algorithm~\ref{alg:homotopy}, $\hat{x}^{(K)}$ denotes 
an approximate solution for minimizing the function $\phi_{\lambda_K}$.
A key idea of the homotopy method is to 
use $\hat{x}^{(K)}$ as the starting point in the proximal gradient method 
for minimizing the next function $\phi_{\lambda_{K+1}}$.
The following lemma shows that if we choose the parameters~$\delta$ 
and~$\eta$ appropriately, then $\hat{x}^{(K)}$  satisfies the approximate
optimality condition for $\lambda_{K+1}$ that guarantees local 
geometric convergence.

\begin{lemma}\label{lem:starting-kkt}
Suppose $\hat x^{(K)}$ satisfies the approximate optimality condition
\[
\omega_{\lambda_K}(\hat x^{(K)}) \leq \delta \lambda_{K} 
\]
for some $\delta<\delta'$.
Let $\lambda_{K+1} = \eta \lambda_K$ for some $\eta$ that satisfies
\begin{equation}\label{eqn:small-eta}
\frac{1+\delta}{1+\delta'} \leq \eta < 1.
\end{equation}
Then we have
\[
\omega_{\lambda_{K+1}}(\hat x^{(K)}) \leq \delta' \lambda_{K+1} .
\]
\end{lemma}
\begin{proof}
If $\omega_{\lambda_K}(\hat x^{(K)}) \leq \delta \lambda_{K}$, 
then there exists $\xi\in\partial\|\hat {x}^{(K)}\|_1$ such that 
$\left\| \nabla\! f(\hat {x}^{(K)})+\lambda_K\xi \right\|_\infty 
\leq \delta \lambda_K$.
Then we have
\begin{eqnarray*}
\omega_{\lambda_{K+1}}(\hat x^{(K)}) 
&\leq& \left\| \nabla\!f(\hat {x}^{(K)})+\lambda_{K+1}\xi \right\|_\infty \\
&=& \left\|\nabla\!f(\hat {x}^{(K)})+\lambda_{K}\xi 
    +(\lambda_{K+1} -\lambda_K)\xi \right\|_\infty \\
&\leq& \left\| \nabla\!f(\hat {x}^{(K)})+\lambda_K\xi \right\|_\infty 
   + |\lambda_{K+1} - \lambda_K| \cdot \|\xi\|_\infty \\
&\leq& \delta\lambda_K + (1-\eta)\lambda_K .
\end{eqnarray*}
Since the condition~(\ref{eqn:small-eta}) implies
$\delta\lambda_K + (1-\eta)\lambda_K \leq \delta'\lambda_{K+1}$,
we have the desired result.
\end{proof}

\begin{lemma}\label{lem:kkt-obj}
Assume that for some $x$ and $\lambda \geq \lambda_\mathrm{tgt}$,
\[
\omega_\lambda(x) \leq \delta' \lambda .
\]
Then for all $\lambda' \in [\lambda_\mathrm{tgt},\lambda]$, we have
\[
\phi_{\lambda'}(x) - \phi_{\lambda'}(x^\star(\lambda'))
\leq 
\frac{2(1+\gamma)(\lambda+\lambda')(\omega_\lambda(x)+\lambda-\lambda')\bar{s}}
     {\rho_-(A, \bar{s}+\tilde{s})}.
\]
 \end{lemma}
\begin{proof}
Let $\xi(\lambda) = \argmin_{\xi\in\partial\|x\|_1} 
\left\| \nabla\!f({x})+\lambda\xi \right\|_\infty$.
Thus
$\omega_\lambda(x)=\left\|\nabla\!f({x})+\lambda\xi(\lambda) \right\|_\infty$.
By the convexity of $\phi_{\lambda'}$, we have
\begin{eqnarray}
\phi_{\lambda'}(x) - \phi_{\lambda'}(x^\star(\lambda'))
&\leq& \langle \nabla\!f({x})+\lambda'\xi(\lambda), x -x^\star(\lambda')\rangle \nonumber\\
&\leq& (\|\nabla\!f({x}) + \lambda\xi(\lambda)\|_\infty + \lambda-\lambda')
\|{x}-x^\star(\lambda')\|_1 \nonumber\\
&=& (\omega_\lambda(x) + \lambda-\lambda') \, \|{x}-x^\star(\lambda')\|_1 .
\label{eqn:obj-kkt-L1-bound}
\end{eqnarray}
By Lemma~\ref{lem:kkt-consequences}, we have
\[
\|\bar{x}-x^\star(\lambda')\|_1 \leq 
(1+\gamma) \sqrt{\bar{s}}\, \|\bar{x}-x^\star(\lambda')\|_2
\leq 
\frac{2(1+\gamma)\lambda'\bar{s}}{\rho_-(A, \bar{s}+\tilde{s})}
\]
and
\[
\|\bar{x}-{x}\|_1 \leq 
(1+\gamma) \sqrt{\bar{s}} \, \|\bar{x}-{x}\|_2
\leq 
\frac{2(1+\gamma)\lambda\bar{s}}{\rho_-(A, \bar{s}+\tilde{s})} .
\]
Therefore, we have
\[
\|{x}-x^\star(\lambda')\|_1 \leq 
\|\bar{x}-{x}\|_1 + \|\bar{x}-x^\star(\lambda')\|_1 \leq 
\frac{2(1+\gamma)(\lambda+\lambda')\bar{s}}{\rho_-(A, \bar{s}+\tilde{s})} .
\]
Now we obtain from (\ref{eqn:obj-kkt-L1-bound}) that
\[
\phi_{\lambda'}({x}) - \phi_{\lambda'}(x^\star(\lambda'))
~\leq~ 
\frac{2 (1+\gamma)(\lambda+\lambda') (\omega_\lambda(x)+\lambda-\lambda')
      \bar{s}}{\rho_-(A, \bar{s}+\tilde{s})} .
\]
This proves the desired result.
\end{proof}

Now we are ready to give an estimate of the overall complexity of
the homotopy method.
First, we need to bound the number of iterations within each call
of Algorithm~\ref{alg:prox-grad}.

Using Lemma~\ref{lem:kkt-grad-mapping}, we can upper bound the
measure for approximate optimality as
\begin{eqnarray*}
\omega_\lambda(x^{(k+1)})
&\leq& \left( 1 + \frac{S_{M_k}(x^{(k)})}{M_k} \right) 
    \bigl\| g_{\lambda,M_k}(x^{(k)}) \bigr\|_2 \\
&\leq& \left( 1 + \frac{\rho_+(A,\bar{s}+2\tilde{s})}{
    \rho_-(A,\bar{s}+2\tilde{s})} \right) 
    \bigl\| g_{\lambda,M_k}(x^{(k)}) \bigr\|_2 \\
&=& (1+\kappa) \bigl\| g_{\lambda,M_k}(x^{(k)}) \bigr\|_2 ,
\end{eqnarray*}
where the second inequality follows from 
\[
S_{M_k}(x^{(k)}) \leq \rho_+(A,\bar{s}+2\tilde{s}), \qquad
M_k \geq \rho_-(A,\bar{s}+2\tilde{s}) ,
\]
which are direct consequences of the line-search termination criterion, 
the restricted smoothness property~(\ref{eqn:restr-smoothness}) and
the restricted strong convexity property~(\ref{eqn:restr-strong-convex}).

In order to bound the norm of $g_{\lambda,M_k}(x^{(k)})$,
we use the inequality~(\ref{eqn:monotone-decrease}) and 
Theorem~\ref{thm:geometric-rate} to obtain
\begin{eqnarray*}
\bigl\|g_{\lambda,M_k}(x^{(k)}) \bigr\|_2^2 
&\leq& 2 M_k \left(\phi_\lambda(x^{(k)}) - \phi_\lambda(x^{(k+1)}) \right) \\
&\leq& 2 M_k \left(\phi_\lambda(x^{(k)}) - \phi_\lambda^\star\right) \\
&\leq& 2 \gammaInc\,\rho_+(A,\bar{s}+2\tilde{s}) 
       \left(1-\frac{1}{4\gammaInc\kappa} \right)^k 
       \left(\phi_\lambda(x^{(0)})-\phi_\lambda^\star\right),
\end{eqnarray*}
where 
$\phi_\lambda^\star = \phi_\lambda(x^\star(\lambda)) = \min_x \phi_\lambda(x)$.
Therefore, in order to satisfy the stopping criteria
\[
\omega_\lambda(x^{(k+1)}) ~\leq~ \delta \lambda,
\]
it suffices to ensure
\[
( 1 + \kappa ) \sqrt{2\gammaInc \rho_+(A,\bar{s}+2\tilde{s})
\left(1-\frac{1}{4\gammaInc\kappa} \right)^k 
\left(\phi_\lambda(x^{(0)})-\phi_\lambda^\star\right)} 
~\leq~ \delta \lambda,
\]
which requires
\[
k ~\geq~ \ln\left( 
\frac{2 \gammaInc (1+\kappa)^2 \rho_+(A,\bar{s}+2\tilde{s}) 
}{\delta^2 \lambda^2} 
\left( \phi_\lambda(x^{(0)}) - \phi_\lambda^\star \right)
\right) \Bigg/
\ln\left( 1 - \frac{1}{4\gammaInc \kappa}\right)^{-1}.
\]

We still need to bound the gap $\phi_\lambda(x^{(0)})-\phi_\lambda^\star$. 
Since Lemma~\ref{lem:starting-kkt} implies that
$\omega_\lambda(x^{(0)}) \leq \delta' \lambda$, we can obtain directly from
Lemma~\ref{lem:kkt-obj} the following inequality by setting $\lambda'=\lambda$ and $x=x^{(0)}$:
\[
\phi_\lambda(x^{(0)})-\phi_\lambda^\star
\leq \frac{4 (1+\gamma) \lambda^2 \bar{s}}{\rho_-(A, \bar{s}+\tilde{s})}.
\]
Therefore, the number of iterations in each call of 
Algorithm~\ref{alg:prox-grad} is no more than
\[
\ln\left( \frac{8\gammaInc (1+\kappa)^2 (1+\gamma) \bar{s}}{\delta^2}
\frac{ \rho_+(A,\bar{s}+2\tilde{s}) }{\rho_-(A,\bar{s}+\tilde{s})} 
\right) \Bigg/
\ln\left( 1 - \frac{1}{4\gammaInc \kappa}\right)^{-1}.
\]
To simplify presentation, we note that
\[
C ~=~ 8 \gammaInc (1+\kappa)^2 (1+\gamma) \bar{s} \kappa
~\geq~ 8 \gammaInc (1+\kappa)^2 (1+\gamma) \bar{s}
\frac{\rho_+(A,\bar{s}+2\tilde{s})}{\rho_-(A,\bar{s}+\tilde{s})} . 
\]
Thus the previous iteration bound is no more than
\[
\ln\left( \frac{C}{\delta^2} \right) \Bigg/
\ln\left( 1 - \frac{1}{4\gammaInc \kappa}\right)^{-1}.
\]
This proves Part~1 of Theorem~\ref{thm:overall-complexity}.
We note that this bound is independent of $\lambda$.

In the homotopy method (Algorithm~\ref{alg:homotopy}), 
after $K$ outer iterations for $K \leq N-1$, we have
from Lemma~\ref{lem:starting-kkt} that
$\omega_{\lambda_{K+1}}(\hat x^{(K)}) \leq \delta' \lambda_{K+1}$. 
The sparse recovery performance bound
\[
\|\hat x^{(K)} - \bar{x}\|_2 \leq 2 \eta^{K+1} \lambda_0 \sqrt{\bar{s}}/\rho_-(A,\bar{s}+\tilde{s}) 
\]
follows directly from Lemma~\ref{lem:kkt-consequences} and $\lambda_{K+1}=\eta^{K+1}\lambda_0$.
Moreover, from Lemma~\ref{lem:kkt-obj} with $\lambda'=\lambda_\mathrm{tgt}$, 
$\lambda=\lambda_{K+1}$, and $x=\hat x^{(K)}$, we obtain
\[
\phi_{\lambda_\mathrm{tgt}}(\hat x^{(K)}) - \phi_{\lambda_\mathrm{tgt}}^\star
\leq 
\frac{4.5 (1+\gamma) \lambda_{K+1}^2 \bar{s}}{\rho_-(A, \bar{s}+\tilde{s})}
= \eta^{2(K+1)} \frac{4.5 (1+\gamma) \lambda_{0}^2 \bar{s}}
       {\rho_-(A, \bar{s}+\tilde{s})}.
\]
This proves Part~2 of Theorem~\ref{thm:overall-complexity}.

In Algorithm~\ref{alg:homotopy}, the number of outer iterations,
excluding the last one for $\lambda_\mathrm{tgt}$, is
\[
N = \left\lfloor
\frac{\ln(\lambda_0/\lambda_\mathrm{tgt})}{\ln(1/\eta)}
\right\rfloor.
\]
The last iteration for $\lambda_\mathrm{tgt}$ uses an absolute precision
$\epsilon$ instead of the relative precision $\delta\lambda_\mathrm{tgt}$.
Therefore, the overall complexity is bounded by
\[
\left(
\frac{\ln(\lambda_0/\lambda_\mathrm{tgt})}{\ln(1/\eta)}
\ln \left( \frac{ C }{\delta^2}\right) 
+ \ln \max\left(1, \frac{ \lambda_\mathrm{tgt}^2 C }{\epsilon^2}\right) 
\right) \Bigg/
\ln\left( 1 - \frac{1}{4\gammaInc \kappa}\right)^{-1} .
\]
Finally, when the PGH method terminates, we have 
$\omega_{\lambda_\mathrm{tgt}}(\hat x^\mathrm{(tgt)}) \leq \epsilon$.
Therefore we can apply Lemma~\ref{lem:kkt-obj} with 
$\lambda=\lambda'=\lambda_\mathrm{tgt}$ and $x=\hat{x}^\mathrm{(tgt)}$ 
to obtain the last desired bound in Part~3.

\section{Numerical experiments}
\label{sec:experiments}

In this section, we present numerical experiments to supports our theoretical
analysis. 
First, we illustrate the numerical properties of the PGH method by comparing
it with several other methods. 
More specifically, we implemented the following methods for solving the 
$\ell_1$-LS problem:
\begin{itemize}
\item PG:  Nesterov's proximal gradient method with adaptive line search
           (Algorithm~\ref{alg:prox-grad}).
\item PGH: our proposed PGH method described in Algorithm~\ref{alg:homotopy}.
\item ADG: Nesterov's accelerated dual gradient method, 
           i.e., Algorithm~(4.9) in \cite{Nesterov07composite}.
\item ADGH: the PGH method in Algorithm~\ref{alg:homotopy}, but with
PG replaced by ADG.
\end{itemize}

We generated a random instance of~(\ref{eqn:l1-LS}) with dimensions 
$m=1000$ and $n=5000$.
The entries of the matrix~$A\in\reals^{m\times n}$ are generated independently
with the uniform distribution over the interval $[-1,+1]$.
The vector~$\bar{x}\in\reals^n$ was generated with the same distribution
at $100$ randomly chosen coordinates (i.e., $\bar{s}=|\supp(\bar{x})|=100$).
The noise~$z\in\reals^m$ is a dense vector with independent random entries
with the uniform distribution over the interval $[-\sigma, \sigma]$, 
where $\sigma$ is the noise magnitude.
Finally the vector $b$ was obtained as $b=A\bar{x}+z$.
In our first experiment, we set $\sigma=0.01$ and choose 
$\lambda_\mathrm{tgt}=1$. 
For this particular instance we have roughly $\|A^T z\|_\infty 0.411$.
To start the PGH method, we have $\lambda_0=\|A^Tb\|_\infty=483.4$.

\begin{figure}[p]
\centering
\psfrag{PG}{\small \hspace{1pt}~~PG}
\psfrag{PGH}{\small \hspace{1pt}~~PGH}
\psfrag{ADG}{\small ADG}
\psfrag{ADGHHH}{\small ADGH}
\psfrag{ADGHH}{\small ADGH}
\psfrag{k}[tc][tc]{$k$}
\psfrag{lambda}[tc][tc]{$\lambda_0/\lambda_K$}
\psfrag{nAx}[tc][tc]{$A$ or $A^T$ multiplications}
\psfrag{Obj}[bc]{$\phi_\lambda(x^{(k)})-\phi_{\lambda_\mathrm{tgt}}^\star$}
\psfrag{OAx}[bc]{$\phi_\lambda(x^{(k)})-\phi_{\lambda_\mathrm{tgt}}^\star$}
\psfrag{NNZ}[bc]{$\|x^{(k)}\|_0$}
\psfrag{Lip}[bc]{$M_k$}
\psfrag{KKT}[bc]{$\omega_\lambda(x^{(k)})$}
\psfrag{Itr}[bc]{}
\vspace{-2ex}
\subfloat[Objective gap.]{
\includegraphics[width=0.45\textwidth]{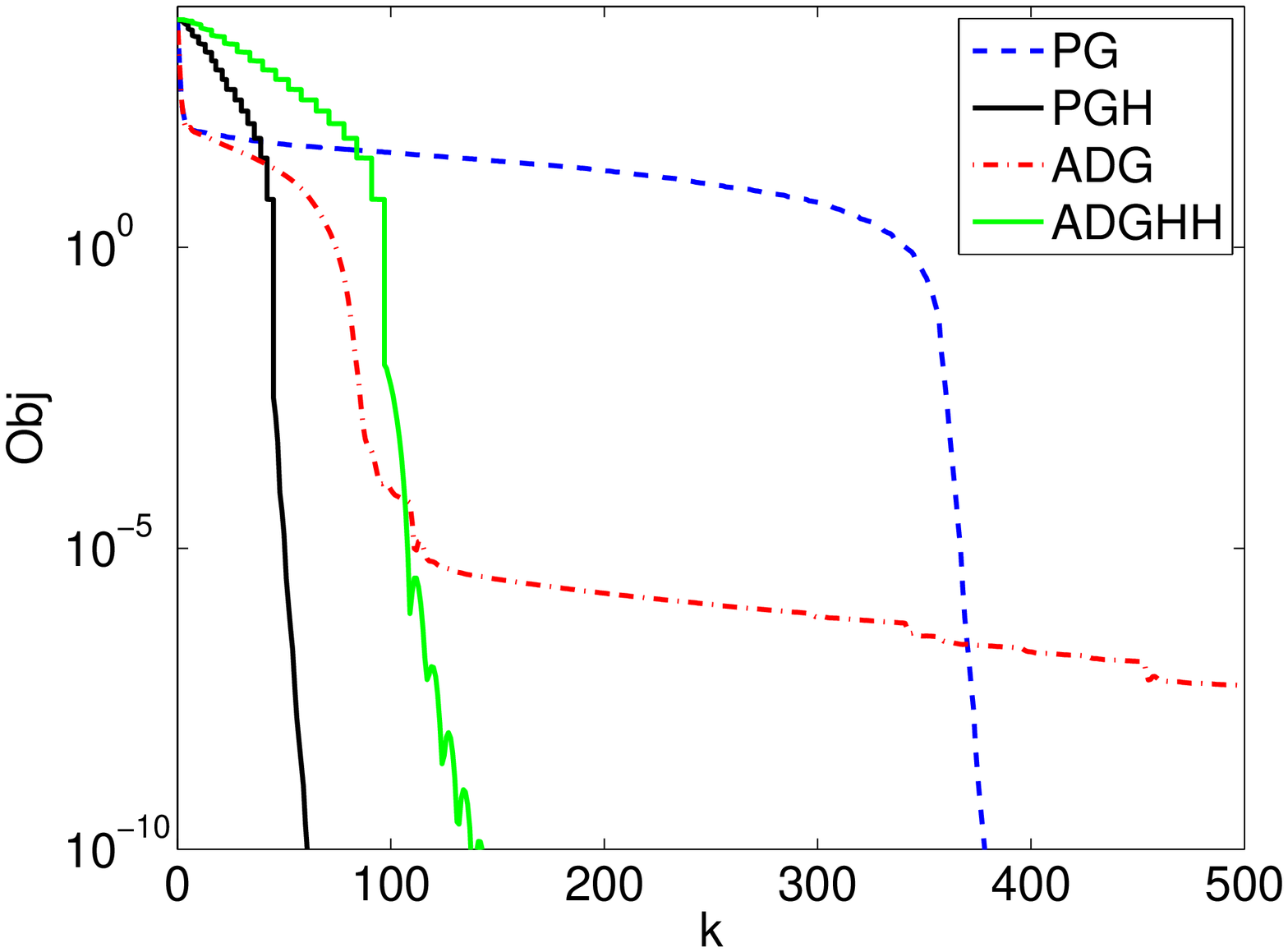}
\label{fig:pgh-obj}}
\hfill
\subfloat[Sparsity along solution path.]{
\includegraphics[width=0.45\textwidth]{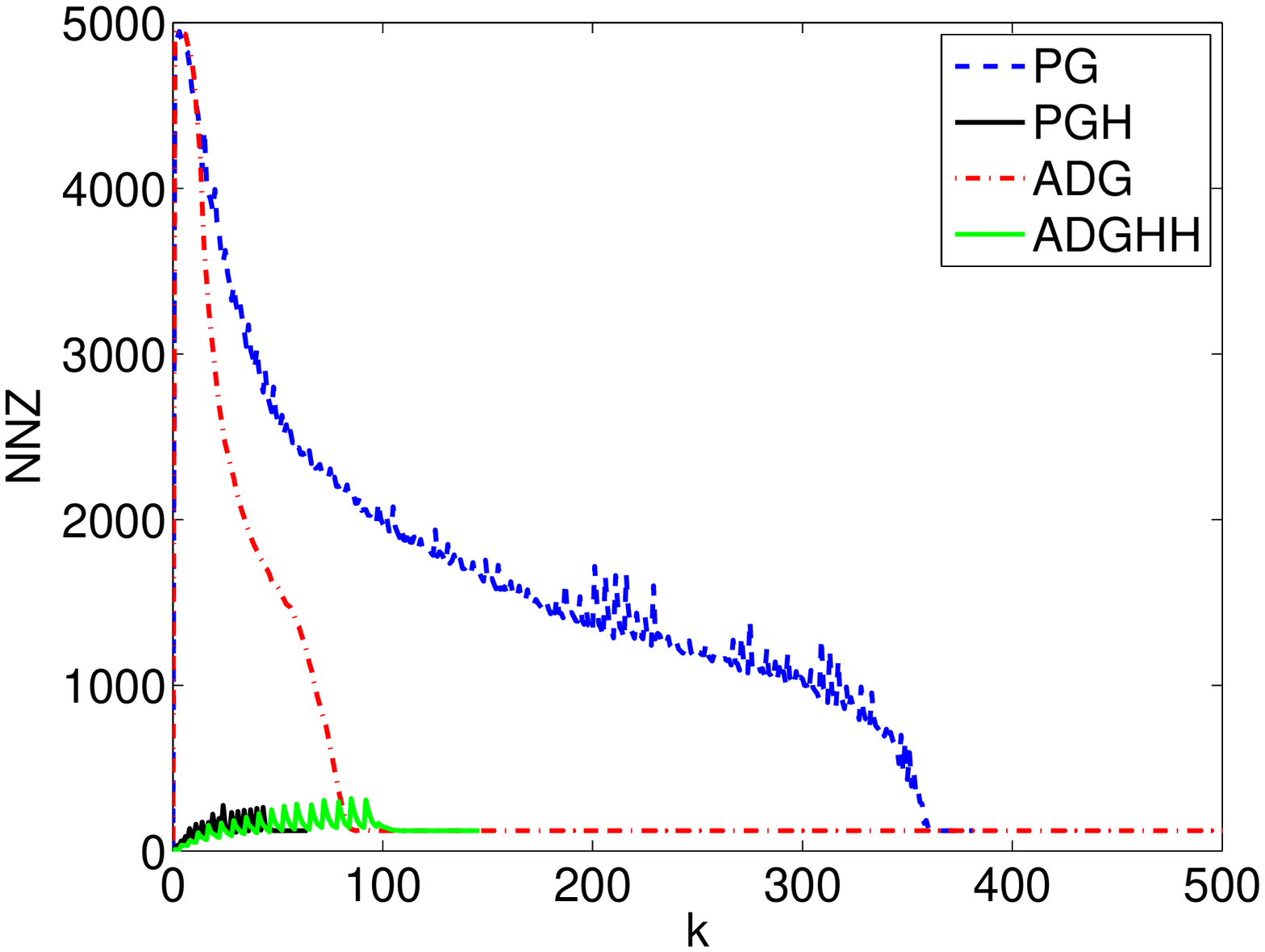}
\label{fig:pgh-nnz}} \\
\subfloat[Optimality residues.]{
\includegraphics[width=0.45\textwidth]{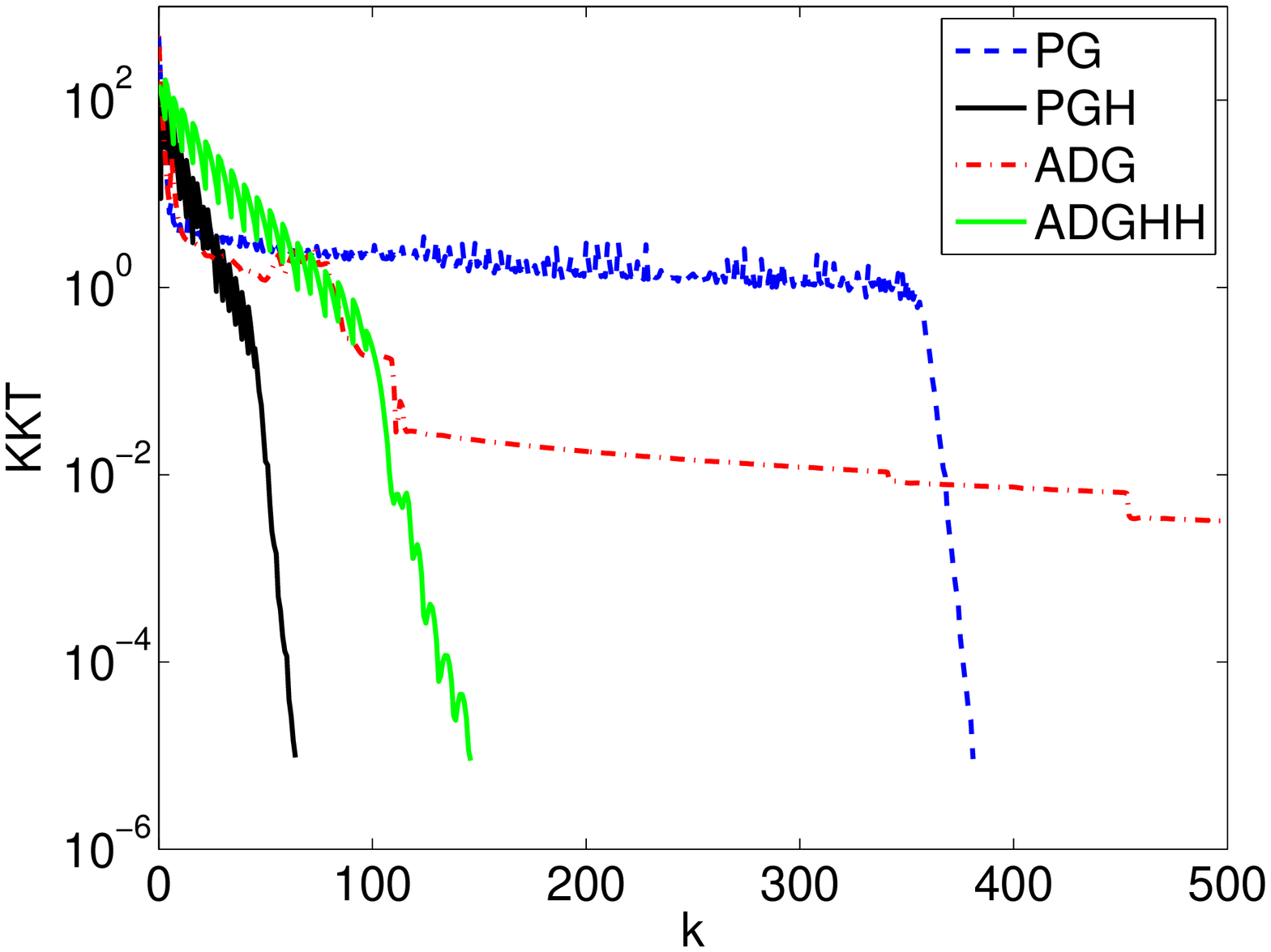}
\label{fig:pgh-res}}
\hfill
\subfloat[Line search results.]{
\includegraphics[width=0.45\textwidth]{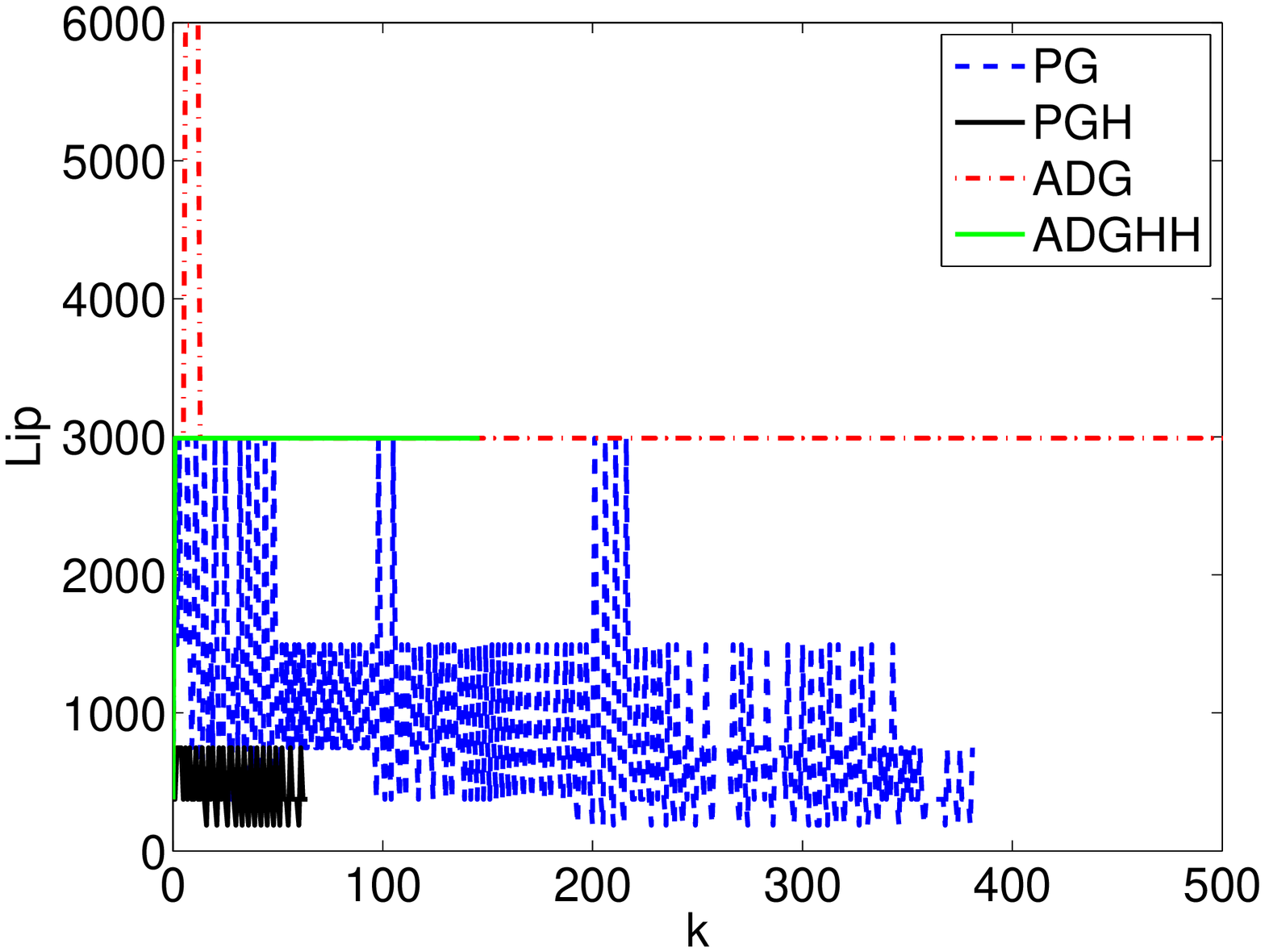}
\label{fig:pgh-Lip}} \\
~\subfloat[Number of iterations for each $\lambda_K$.]{
\includegraphics[width=0.43\textwidth]{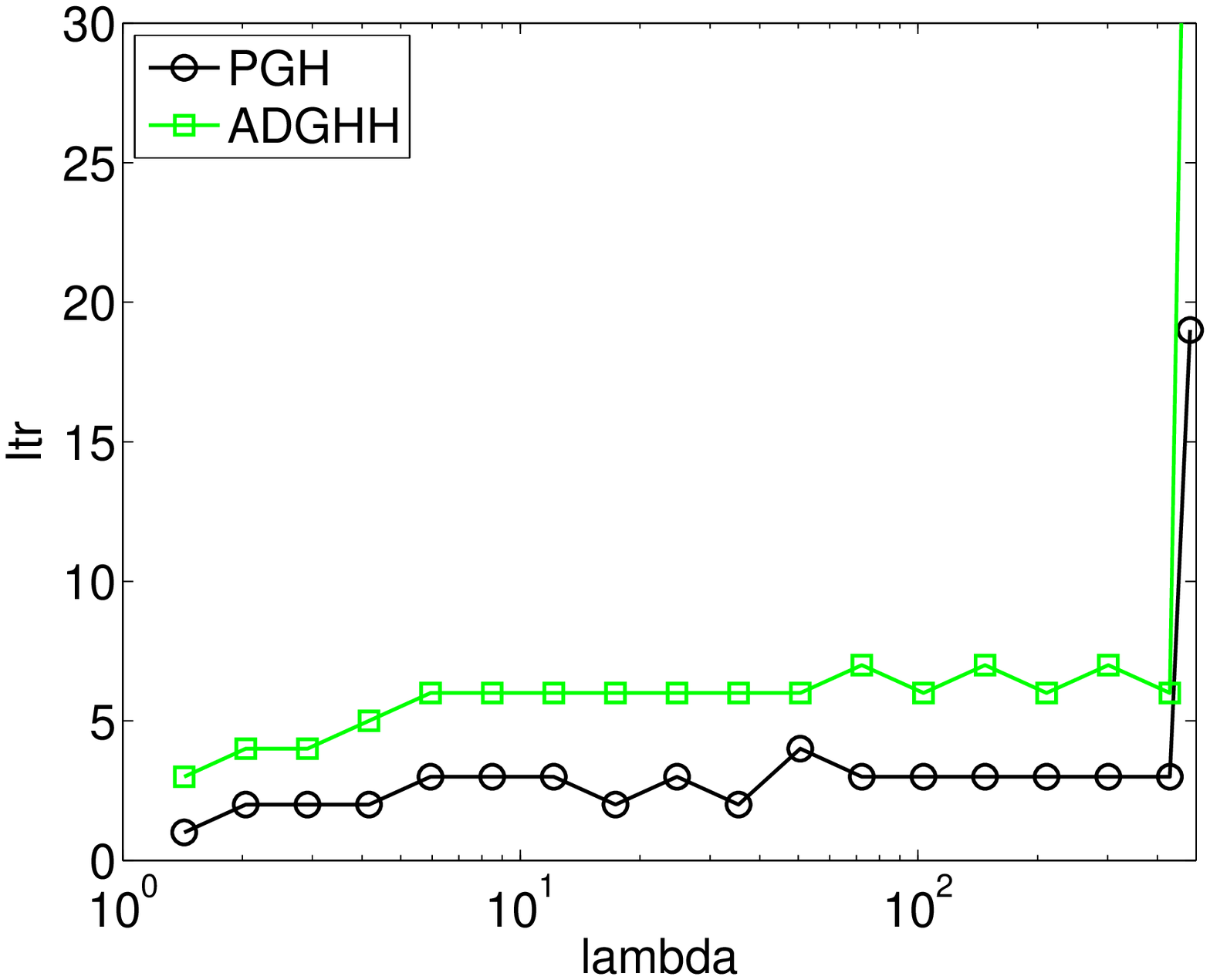}
\label{fig:pgh-itr}}
\hfill
\subfloat[Number of matrix-vector multiplications.]{
\includegraphics[width=0.45\textwidth]{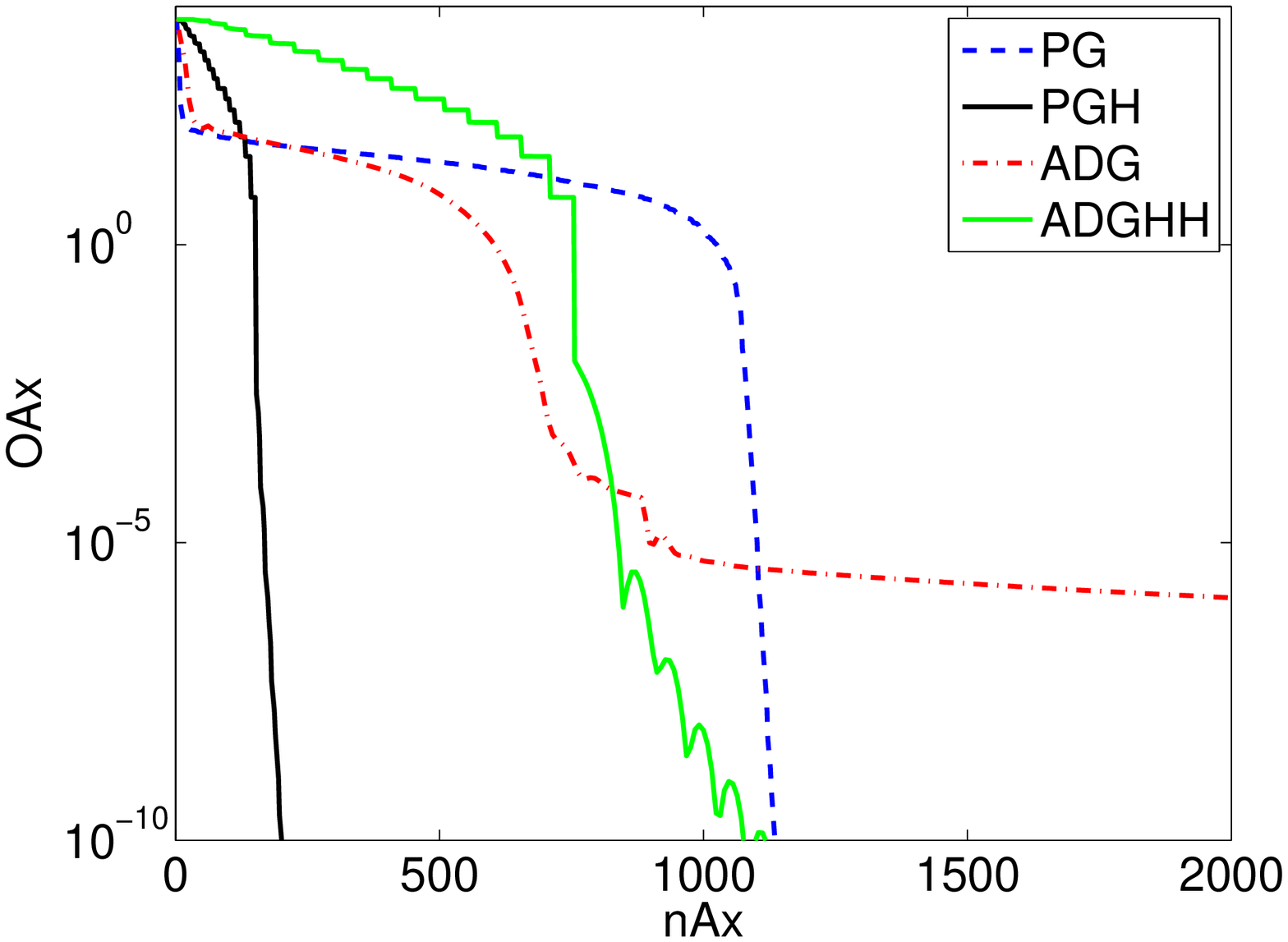}
\label{fig:pgh-mvc}}
\caption{Solving a random instance of the $\ell_1$-LS problem.
Problem sizes: $m=1000$, $n=5000$, $\bar{s}=100$, and $\lambda_\mathrm{tgt}=1$.
Entries of $A\in\reals^{m\times n}$ were generated with independent
uniform distributions over $[-1,+1]$, and $\|z\|_\infty=0.01$.
Algorithmic parameters: 
$\gammaInc=2$, $\gammaDec=2$, $\eta=0.7$, and $\delta=0.2$.}
\label{fig:pgh}
\end{figure}

Figure~\ref{fig:pgh} illustrates various numerical properties
of the four different methods for solving this random instance.
We used the parameters $\gammaInc=2$ and $\gammaDec=2$ in all four methods.
For the two homotopy methods (whose acronyms end with the letter~H), 
we used the parameters $\eta=0.7$ and $\delta=0.2$.
In the first four subfigures~\subref{fig:pgh-obj}-\subref{fig:pgh-Lip}, 
the horizontal axes show the cumulative count of inner iterations 
(total number of proximal-gradient steps).
For the two homotopy methods, the vertical line segments in the 
subfigures~\subref{fig:pgh-obj}, \subref{fig:pgh-res} and~\subref{fig:pgh-mvc}
indicate switchings of homotopy stages 
(when the value of~$\lambda$ is reduced by the factor~$\eta$) --- 
they reflect the change of objective function or the optimality residue
for the same vector~$x^{(k)}$.

Figure~\ref{fig:pgh}\subref{fig:pgh-obj} shows the objective gap 
$\phi_\lambda(x^{(k)})-\phi_{\lambda_\mathrm{tgt}}^\star$ versus the total
number of iterations~$k$.
The PG method solves the problem with the target regularization parameter
$\lambda_\mathrm{tgt}$ directly. 
For the first 350 or so iterations, it demonstrated a slow sublinear 
convergence rate (theoretically $O(1/k)$), but converged rapidly for the last
30 iterations with a linear rate.
Referring to Figure~\ref{fig:pgh}\subref{fig:pgh-nnz}, we see that the slow 
convergence phase of PG is associated with relatively dense iterates 
(with $\|x^{(k)}\|_0$ ranging from 5,000 to several hundreds), while the fast
linear convergence in the end coincides with sparse iterates with
$\|x^{(k)}\|_0$ around one hundred. 
In contrast, the PGH method maintains sparse iterates (always less than 300)
along the whole solution path, and demonstrates geometric convergence
at each stage of homotopy continuation. 

Figure~\ref{fig:pgh}\subref{fig:pgh-res} shows the optimality residues of 
different methods versus the number of iterations~$k$. 
They demonstrate similar trends as the objective function gap, but clearly they
oscillate along the solution path and do not decrease monotonically.
Figure~\ref{fig:pgh}\subref{fig:pgh-Lip} plots the local Lipschitz constants
returned by the line search procedure at each iteration.
We see that the adaptive line-search method settles with much smaller $M_k$ 
when the iterates are sparse.
There is a striking similarity between the final stages of the PG method 
and the PGH method. 
However, the PGH method avoids the slow sublinear convergence by
maintaining sparse iterates along its whole solution path.

Also plotted in Figure~\ref{fig:pgh} are numerical characteristics of the 
ADG and ADGH methods. 
We see that the ADG method is much faster than the PG method in the early
phase, which can be explained by its better convergence rate, i.e., 
$O(1/k^2)$ instead of $O(1/k)$ for PG. 
However, it stays with the sublinear rate even when the iterates $x^{(k)}$
becomes very sparse. The reason is that ADG cannot automatically exploit the 
local strong convexity as PG does, so it eventually lagged behind when the 
iterates became very sparse 
(see discussions in \cite{Nesterov07composite}). 
In the method ADGH, we combine the homotopy continuation strategy with the 
ADG method. 
It improves a lot compared with ADG, but still does not have  
linear convergence and thus is much slower than the PGH method. 

Figure~\ref{fig:pgh}\subref{fig:pgh-itr} shows the number of proximal-gradient 
steps performed at each homotopy stage (corresponding to each~$\lambda_K$) 
of the two homotopy methods.
We see that the final stage of the PGH method took 19 inner iterations to 
reach the absolute precision $\epsilon=10^{-5}$, 
and all earlier stages took only 1 to 4 inner iterations to reach the relative
precision $\delta\lambda_K$.
We note that the number of inner iterations at each intermediate stage
stayed relatively constant, even though the tolerance for the optimality
residue decreases as $\delta\lambda_k = \eta^K \delta \lambda_0$. 
This is predicted by Part~1 of Theorem~\ref{thm:overall-complexity}.
The ADGH method, which employs the ADG method for solving
each stage, took more number of inner iterations at each stage.
This again reflects its lack of capability of exploiting the restricted 
strong convexity.

The number of inner iterations is not the whole story for evaluating the
performance of the algorithms. 
Figure~\ref{fig:pgh}\subref{fig:pgh-mvc} shows the objective gap versus 
the total number of matrix-vector multiplications with either~$A$ or~$A^T$.
Evaluating the objective function $f(x^{(k)})$ costs one matrix-vector 
multiplication, and evaluating the gradient $\nabla\!f(x^{(k)})$ costs an 
additional multiplication. The estimate in~(\ref{eqn:line-search-bound}) 
states that each proximal-gradient step in the PG method needs on average 
two calls of the oracle. But one of them is done in the line search procedure,
and it requires only the function value.
Therefore each inner iteration on average costs roughly three matrix-vector 
multiplications. 
On the other hand, each iteration of the ADG method on average costs eight
matrix-vector multiplications \cite{Nesterov07composite}.
These factors are confirmed by comparing the horizontal scales of
the Figures~\ref{fig:pgh}\subref{fig:pgh-obj} 
and~\ref{fig:pgh}\subref{fig:pgh-mvc}.
We found that the number of matrix-vector multiplications is a very
precise indicator for the running time of each algorithm.
From this perspective, the advantage of the PGH method is more pronounced.

\begin{figure}[t]
\centering
\psfrag{DELTA=1}{\small $\delta=0.1$}
\psfrag{DELTA=2}{\small $\delta=0.2$}
\psfrag{DELTA=3}{\small $\delta=0.8$}
\psfrag{k}[tc][tc]{$k$}
\psfrag{Obj}[bc]{$\phi_\lambda(x^{(k)})-\phi_{\lambda_\mathrm{tgt}}^\star$}
\psfrag{NNZ}[bc]{$\|x^{(k)}\|_0$}
\psfrag{tObj}[bc]{(a) Objective gap}
\psfrag{tNNZ}[bc]{(b) Number of non-zeros}
\vspace{-2ex}
\subfloat[Objective gap.]{
\includegraphics[width=0.45\textwidth]{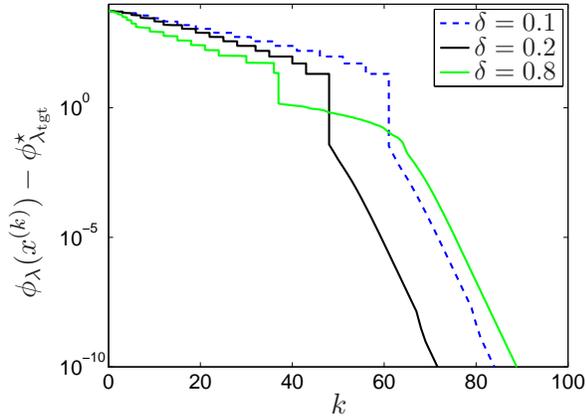}
\label{fig:delta-obj}}
\hfill
\subfloat[Sparsity along solution path.]{
\includegraphics[width=0.45\textwidth]{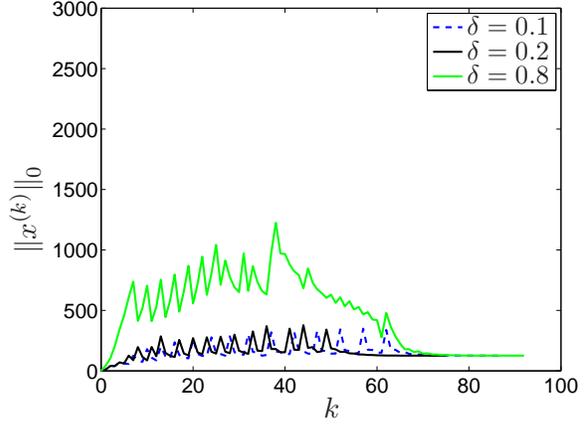}
\label{fig:delta-nnz}}
\caption{Performance of the PGH method by varying $\delta$  
while keeping $\eta=0.7$.}
\label{fig:delta}
\end{figure}

\begin{figure}[t]
\centering
\psfrag{ETA===1}{\small $\eta=0.2$}
\psfrag{ETA===2}{\small $\eta=0.5$}
\psfrag{ETA===3}{\small $\eta=0.8$}
\psfrag{k}[tc][tc]{$k$}
\psfrag{Obj}[bc]{$\phi_\lambda(x^{(k)})-\phi_{\lambda_\mathrm{tgt}}^\star$}
\psfrag{NNZ}[bc]{$\|x^{(k)}\|_0$}
\psfrag{tObj}[bc]{(a) Objective gap}
\psfrag{tNNZ}[bc]{(b) Number of non-zeros}
\subfloat[Objective gap.]{
\includegraphics[width=0.45\textwidth]{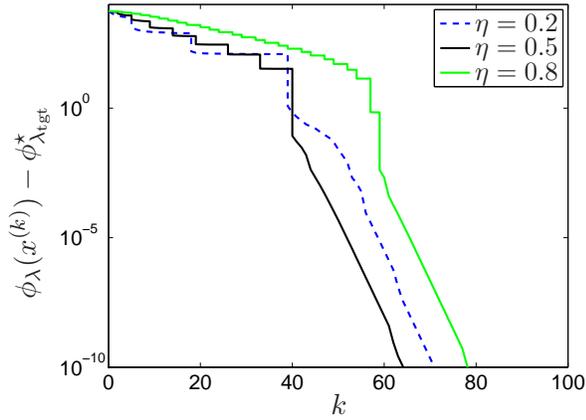}
\label{fig:eta-obj}}
\hfill
\subfloat[Sparsity along solution path.]{
\includegraphics[width=0.45\textwidth]{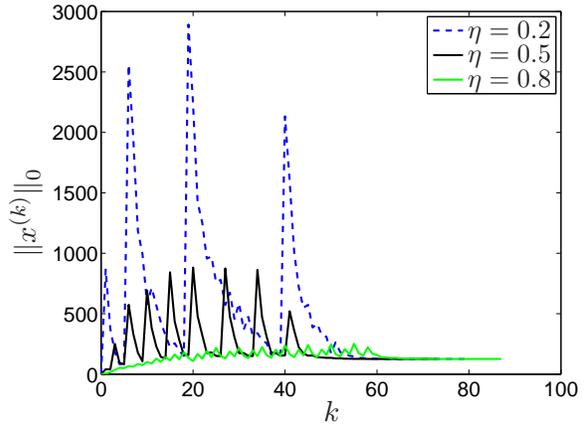}
\label{fig:eta-nnz}}
\caption{Performance of the PGH method by varying $\eta$ 
while keeping $\delta=0.2$.}
\label{fig:eta}
\end{figure}

Next we conducted experiments to test the sensitivity of the PGH method
with respect to the choices of parameters~$\delta$ and~$\eta$.
Figure~\ref{fig:delta} shows the objective gap and sparsity of
the iterates along the solution path for different~$\delta$ while
keeping $\eta=0.7$.
We see that when $\delta$ is reduced from~$0.2$ to~$0.1$, the iterates
became slightly more sparse, hence the convergence rate at each stage can be 
slightly faster due to better conditioning. 
However, this was countered by more iterations at each stage required by
reaching more stringent precision, and the overall number of proximal-gradient
steps increased.
On the other hand, increasing~$\delta$ to~$0.8$ made the intermediate stages
faster by requiring loose precision. 
However, this comes at the cost of less sparse iterates, and the final stage
suffers a slow sublinear convergence in the beginning.

Figure~\ref{fig:eta} shows the numerical behaviors of the PGH method
by varying~$\eta$ while keeping $\delta=0.2$.
We see relatively big variations of the sparsity of the iterates, but these
did not affect much of the overall iteration count.
The intermediate stages may suffer from slow convergence with less sparsity,
but they only need to be solved to a very rough precision.
It is more important to start the last stage with a sparse vector and
enjoy the fast convergence to the final precision.
It is interesting to note that the sufficient conditions
$(1+\delta)/(1+\delta')\leq\eta<1$ (in Theorem~\ref{thm:overall-complexity})
and $0<\delta<\delta'<1$ implies $\eta>0.5$.
But we see that a more aggressive $\eta=0.2$ still works well for this instance.

\subsection{Comparison with SpaRSA and FPC}

\begin{figure}[t]
\centering
\psfrag{PGH eta === 0.7}{\small PGH $\eta=0.7$}
\psfrag{PGH eta === 0.2}{\small PGH $\eta=0.2$}
\psfrag{SpaRSA}{\small SpaRSA-MC}
\psfrag{FPC}{\small FPC-BB}
\psfrag{k}[tc][tc]{$k$}
\psfrag{lambda}[tc][tc]{$\lambda_0/\lambda_K$}
\psfrag{Obj}[bc]{$\phi_\lambda(x^{(k)})-\phi_{\lambda_\mathrm{tgt}}^\star$}
\psfrag{itr}[bc]{}
\vspace{-2ex}
\subfloat[Objective gap.]{
\includegraphics[width=0.46\textwidth]{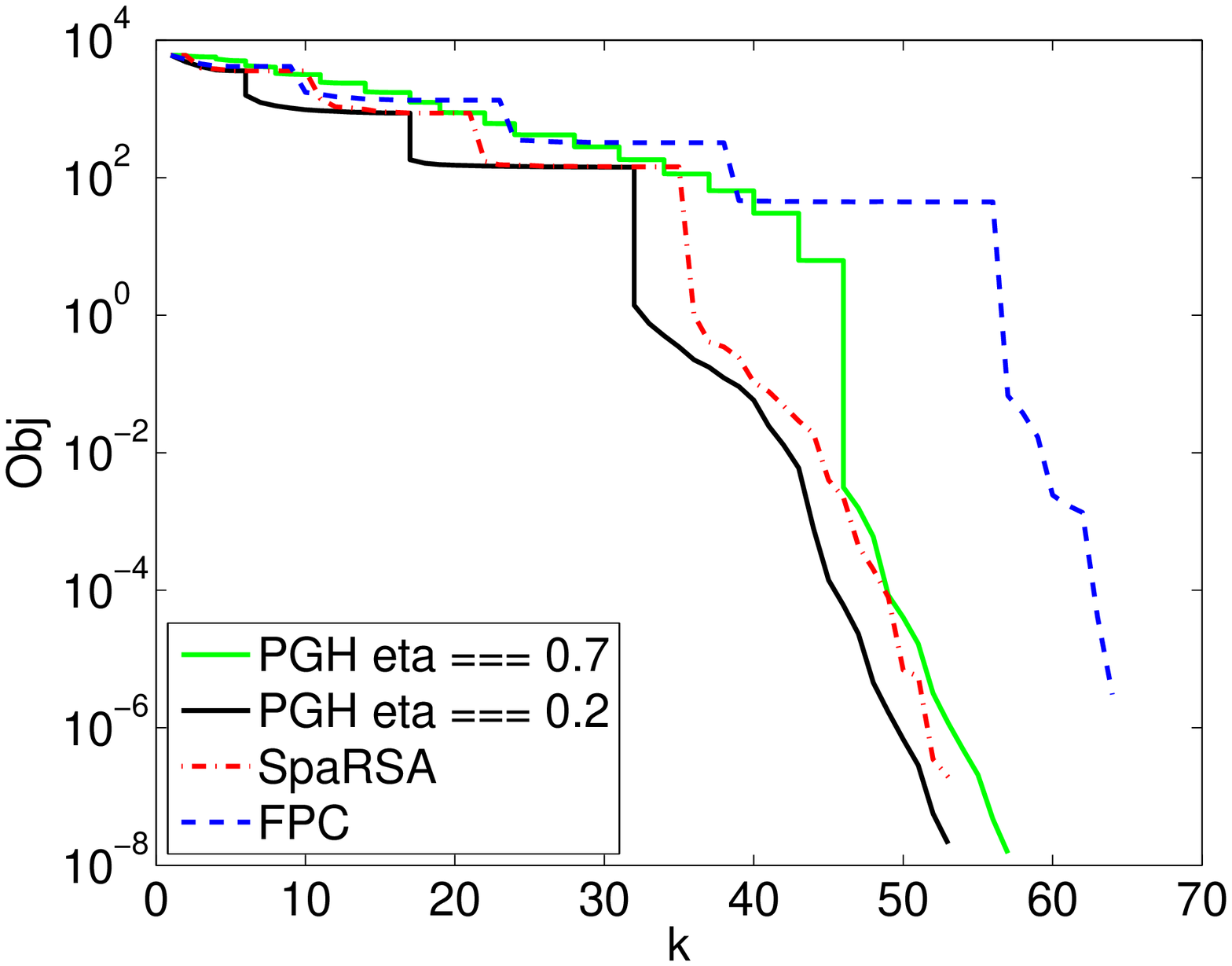}
\label{fig:fpc-sparsa-obj}}
\hfill
\subfloat[Number of inner iterations for each $\lambda_K$.]{
\includegraphics[width=0.46\textwidth]{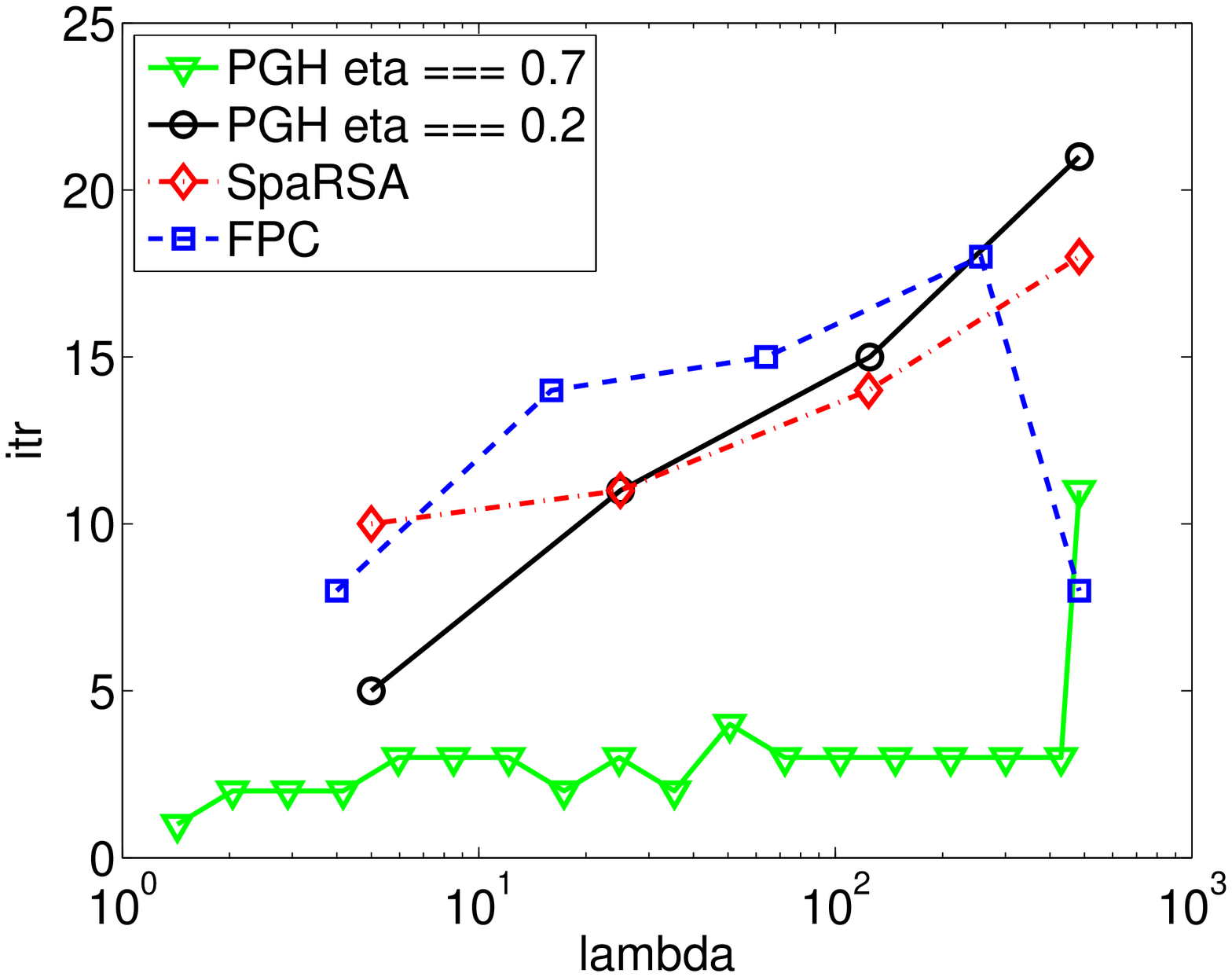}
\label{fig:fpc-sparsa-itr}}
\caption{Comparison with SpaRSA and FPC.}
\label{fig:fpc-sparsa}
\end{figure}

As mentioned in the introduction, similar approximate homotopy/continuation 
methods have been studied for the $\ell_1$-LS problem. 
Here we compare the PGH method with two most relevant ones:
sparse reconstruction by separable approximation (SpaRSA) \cite{WrightNF09}, 
and fixed point continuation (FPC) \cite{HaleYinZhang08}.
In particular, the same proximal gradient method~(\ref{eqn:prox-grad-step}) 
is used in each iteration of both SpaRSA and FPC.
Their continuation strategies are both based on reducing $\lambda$ by a 
constant factor at each stage.

SpaRSA uses Barzilai-Borwein (spectral) method for choosing~$L_k$
at each step.
More specifically, at each iteration the parameter~$L_k$ is initialized as
\[
L_k = \frac{\left\|A \left(x^{(k)}-x^{(k-1)}\right)\right\|_2^2}
{\|x^{(k)}-x^{(k-1)}\|_2^2},
\]
then it is increased by a constant factor until an acceptance criterion
is satisfied.
When both $x^{(k)}$ and $x^{(k-1)}$ are sparse, say 
$|\supp(x^{(k)}) \cup \supp(x^{(k-1)})| \leq s$ for some integer~$s$, then 
the above $L_k$ satisfies
\[
\rho_-(A,s) \leq L_k \leq \rho_+(A,s).
\]
According to Section~\ref{sec:restr-eig-cond}, such a line search method
is able to exploit the restricted strong convexity, similar as the PGH
method. However, the line-search acceptance criterion of SpaRSA is different
from PGH, and they also have different stopping criteria for each 
homotopy stage. 
Global geometric convergence of either SpaRSA or FPC has not been established.

In our numerical experiments, we used the monotone version of SpaRSA with 
continuation, which we call SpaRSA-MC.
For FPC, we used a more recent implementation by the authors of 
\cite{HaleYinZhang08} that also employs Barzilai-Borwein line search, 
which is called FPC-BB.
In fact FPC-BB solves the equivalent problem
\[
\minimize_x \quad \|x\|_1 + \frac{\mu}{2} \|Ax-b\|_2^2
\]
where $\mu=1/\lambda$. 
Moreover, it further scales the matrix~$A$ so that the maximum singular value 
is at most~$1$.
In Figures~\ref{fig:fpc-sparsa} and~\ref{fig:non-sparse}, the
results of FPC-BB are plotted after we reversed the scalings
in order to compare with other methods. 
Default options were used in both methods. 
SpaRSA-MC reduces the value of~$\lambda$ roughly with an factor $\eta=0.2$,
and FPC-BB has an equivalent factor $\eta=0.25$. 
For meaningful comparison, we also present the results for PGH with $\eta=0.2$,
in addition to its default value $\eta=0.7$.
The same relative precision $\delta=0.2$ was used in both cases for PGH. 

Figure~\ref{fig:fpc-sparsa} shows the numerical results of different algorithms
on the same random instance studied in Figure~\ref{fig:pgh}.
They demonstrate similar numerical properties, and SpaRSA-MC is especially 
similar to PGH with $\eta=0.2$. 
The numbers of iterations at each continuation stage depend on the
specific stopping criteria used in different algorithms. 
In Figure~\ref{fig:fpc-sparsa}\subref{fig:fpc-sparsa-itr},
the small number of iterations in the final stage of FPC-BB is a result
of the relatively loose precision specified in its default options, 
which is also reflected in 
Figure~\ref{fig:fpc-sparsa}\subref{fig:fpc-sparsa-obj}.
According to Figure~\ref{fig:eta}\subref{fig:eta-nnz},
the aggressive decreasing factors~$\eta$ used in SpaRSA and FPC  
can lead to less sparse iterates along the solution path, thus relatively
slower convergence at the intermediate stages.
But their overall iteration counts are comparable to PGH with $\eta=0.7$. 

\begin{figure}[t]
\centering
\psfrag{PG}{\small PG}
\psfrag{PGH eta = 0.7}{\small PGH $\eta=0.7$}
\psfrag{SpaRSA-M}{\small SpaRSA-M}
\psfrag{SpaRSA-MC}{\small SpaRSA-MC}
\psfrag{FPC}{\small FPC-BB}
\psfrag{FPC2}{\small FPC-BB-HA}
\psfrag{ADG}{\small ADG}
\psfrag{ADGHH}{\small ADGH}
\psfrag{k}[tc][tc]{$k$}
\psfrag{obj}[bc]{$\phi_\lambda(x^{(k)})-\phi_\mathrm{tgt}^\star$}
\psfrag{itr}[bc]{number of inner iterations}
\includegraphics[width=0.66\textwidth]{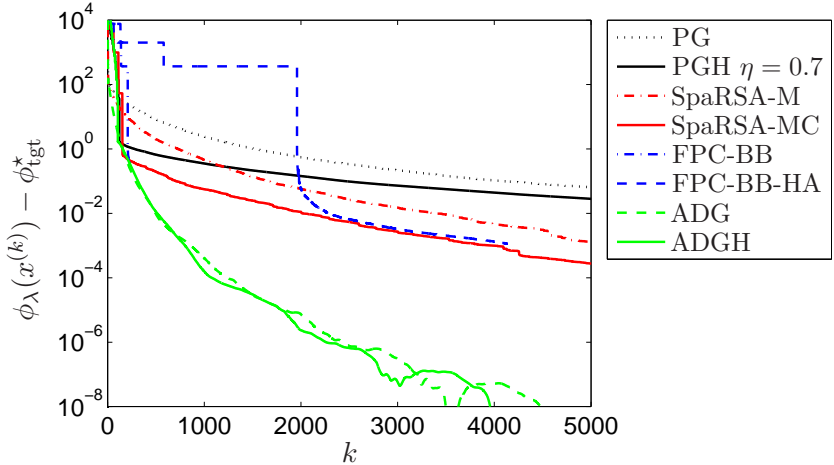}
\caption{Comparison of different methods for solving a non-sparse random
instance.}
\label{fig:non-sparse}
\end{figure}

We also conducted experiments with random problem instances where the 
vector $\bar{x}$ is not sufficiently sparse. 
Figure~\ref{fig:non-sparse} shows the objective gap of different methods
when solving a random problem instance generated similarly 
as the one studied in Figure~\ref{fig:pgh}.
The only difference is that here the vector $\bar{x}$ has $500$ nonzero 
elements. 
In this case, all methods demonstrate sub-linear convergence. 
SpaRSA-M is the monotone version of SpaRSA without continuation.
FPC-BB terminated prematurely because its default accuracy for its stopping
criterion is too low. 
FPC-BB-HA is the result after we set a much higher accuracy in calling the
FPC-BB method. 
It looks that the same higher accuracy is used in all the homotopy stages,
so the number of inner iterations increased for each stage. 
We see that the algorithms with homotopy continuation still perform better
than their single-stage counterparts, but the improvements are less impressive. 
Instead, the accelerated gradient methods ADG and ADGH 
outperform other methods by a big margin. 

\subsection{Basis pursuit}

\begin{figure}[t]
\centering
\psfrag{PGH eta === 0.7}{\small PGH $\eta=0.7$}
\psfrag{PGH eta === 0.2}{\small PGH $\eta=0.2$}
\psfrag{SpaRSA}{\small SpaRSA-MC}
\psfrag{FPC}{\small FPC-BB}
\psfrag{k}[tc][cl]{$k$}
\psfrag{Obj}[bc]{$\phi_\lambda(x^{(k)})$}
\psfrag{rcv}[bc]{$\|x^{(k)}-\bar{x}\|_2$}
\psfrag{lambda}[cc]{$\lambda_0/\lambda_K$}
\subfloat[Objective value. 
          Note that $\phi_\lambda^\star\to 0$ as $\lambda\to 0$.]{
\includegraphics[width=0.48\textwidth]{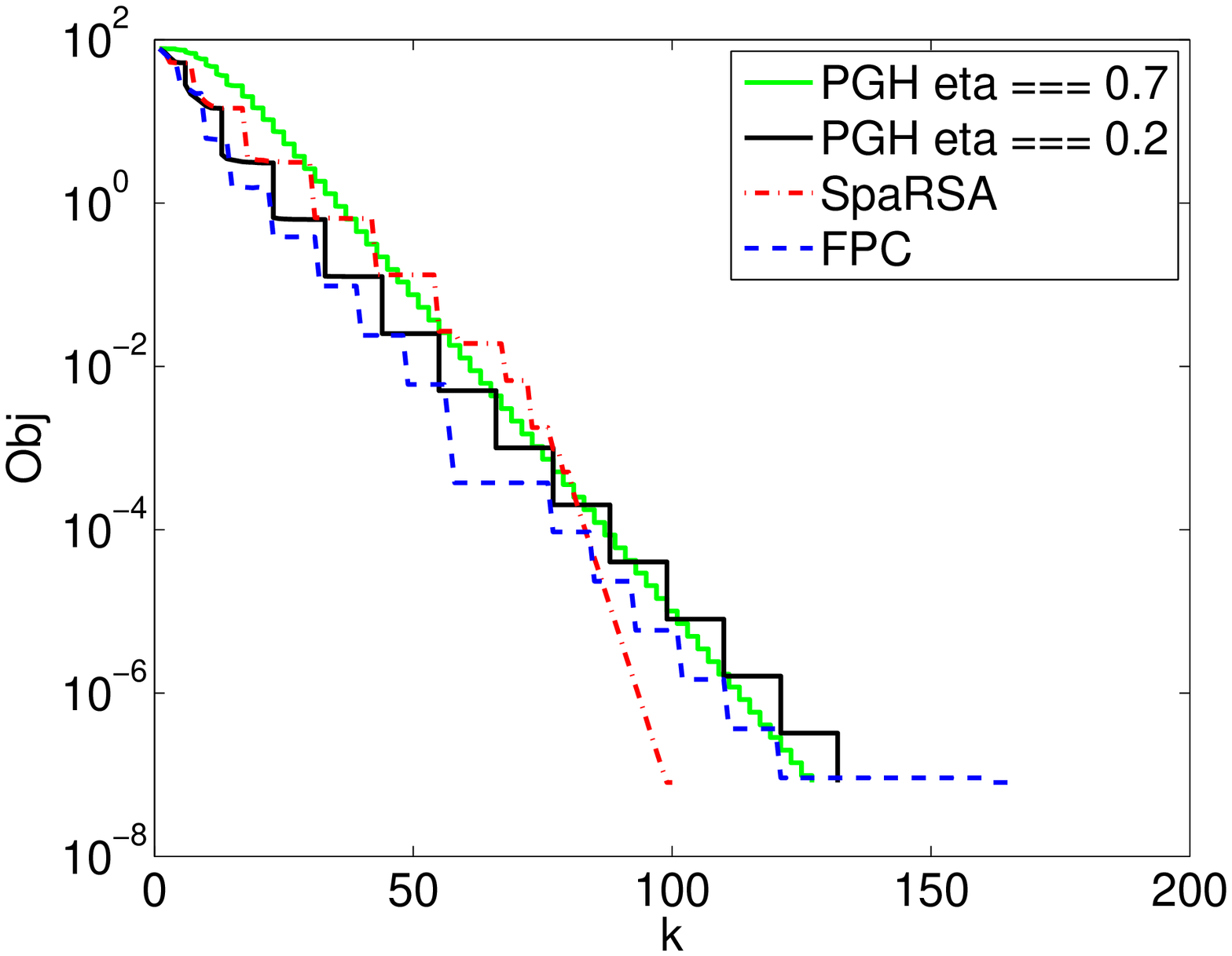}
\label{fig:fft-obj}
}
\hfill
\subfloat[Recovery error.]{
\includegraphics[width=0.48\textwidth]{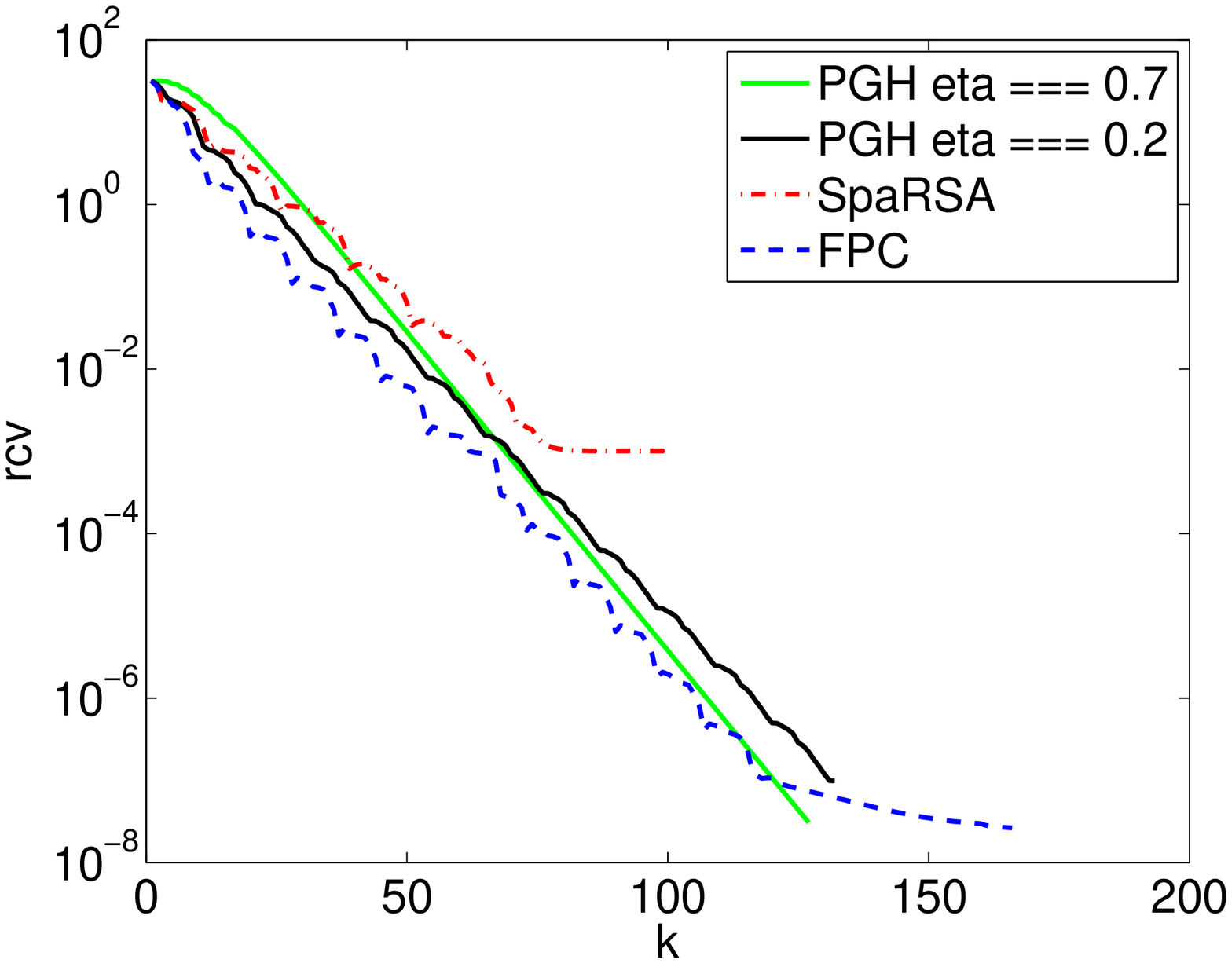}
\label{fig:fft-rcv}
} \\
\subfloat[Number of inner iterations]{
\includegraphics[width=0.48\textwidth]{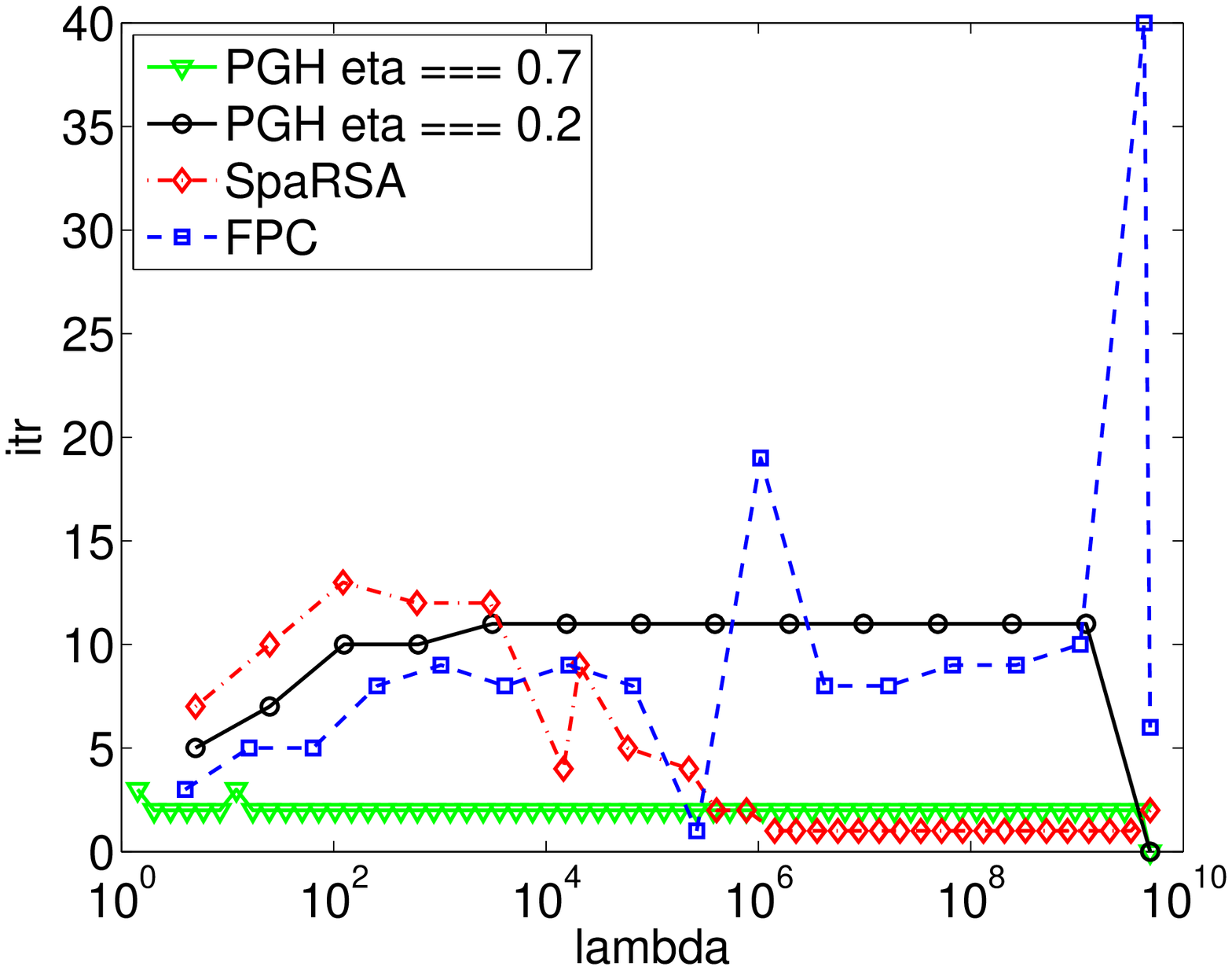}
\label{fig:fft-itr} 
}
\caption{Basis pursuit via homotopy continuation: 
an example with partial FFT matrix.}
\label{fig:fft}
\end{figure}

Finally we present an experiment of solving the basis pursuit (BP) 
problem~(\ref{eqn:BP}) using PGH, and compare it with FPC and SpaRSA.
In this experiment, the matrix~$A$ is a partial FFT matrix. 
More specifically, we choose $m=10,000$ rows at random from the $n\times n$ 
FFT matrix with $n=2^{16}=65536$. 
The vector $\bar{x}\in\reals^n$ has nonzero entries at only $\bar{s}=1000$ 
randomly chosen coordinates, and they were generated independently from the 
normal distribution with zero mean and unit variance. 
Then we set $b=A\bar{x}$ in the BP problem 
(i.e., this is the noise-free case with $z=0$).

In this case, since $A$ is a matrix with complex numbers, we need to 
replace all the real transpose in the algorithms with Hermitian transpose,
and replace the soft-thresholding operator in~(\ref{eqn:soft}) with 
\[
\SoftThresholding(x_i,\alpha) 
=\frac{\max\{|x_i|-\alpha, 0\}}{\max\{|x_i|-\alpha, 0\} + \alpha},
\]
where $|x_i|$ denotes the modulus of the complex number~$x_i$
\cite{WrightNF09}.

The solution to the BP problem~(\ref{eqn:BP}) can be obtained by letting
$\lambda\to 0$ in the $\ell_1$-LS problem~(\ref{eqn:l1-LS}).
In order to use the PGH method, we set $\lambda_\mathrm{tgt}=10^{-10}$.
The same parameter was also used in calling SpaRSA-MC and FPC-BB.
Figure~\ref{fig:fft} shows the numerical results. 
Again we observe remarkable resemblance between these methods 
in Figure~\ref{fig:fft}\subref{fig:fft-obj}.
However, in Figure~\ref{fig:fft}\subref{fig:fft-rcv}, we see the recovery
error of SpaRSA-MC stayed at the level $10^{-3}$ while its objective 
function in Figure~\ref{fig:fft}\subref{fig:fft-obj} converged to zero
faster than other methods.
The reason is that SpaRSA has a fixed accuracy requirement for all continuation
stages except for the last one.
As shown in Figure~\ref{fig:fft}\subref{fig:fft-itr}, 
when $\lambda_K$ becomes very small, this constant accuracy is always reached 
within one iteration, and such a low accuracy is too loose for the algorithm to 
track the homotopy path closely.
Therefore, even though the objective function converges to zero quickly,
the recovery error stayed large.
This is also confirmed through the denser continuation stages in the 
second half of SpaRSA-MC, as shown in Figure~\ref{fig:fft}\subref{fig:fft-itr}.
To see this, we note that the adaptive continuation used in SpaRSA is
\[
\lambda_{K+1} 
=\max\left\{\eta\|A^T(A\hat{x}^{(K)}-b)\|_\infty, ~\lambda_\textrm{tgt}\right\}.
\]
If~$\hat{x}^{(K)}$ is an accurate solution for the stage~$\lambda_K$, then
we have $\|A^T(A\hat{x}^{(K)}-b)\|_\infty\approx\lambda_K$ and thus
$\lambda_{K+1}\approx\eta\lambda_K$ with $\eta=0.2$ as the default value.
When this is not the case, then  
$\|A^T(A\hat{x}^{(K)}-b)\|_\infty$ can be notably larger than $\lambda_K$,
and thus the regularization parameter reduces at a much slower pace.
Similar as PGH, FPC-BB sets the accuracy for each continuation stage
to be proportional to the regularization parameter, but for a different
stopping criterion. 
With our choice of stopping criterion, 
$\omega_\lambda(\hat{x}^{(K)})\leq \delta\lambda_K$,
the number of inner iterations for each continuation stage stayed 
roughly constant along the homotopy path.

This example also demonstrates the advantage of PGH and other approximate 
homotopy continuation methods over the exact homotopy path-following 
methods \cite{OsbornePT00a, OsbornePT00b,EfronHJT04}.
Figure~\ref{fig:fft}\subref{fig:fft-rcv} shows that high-precision 
recovery can be obtained by PGH in less than $150$ iterations 
(which corresponds to roughly $450$ matrix-vector multiplications).
This is much more efficient than using the exact homotopy path-following 
methods, which need to track at least $1000$ breakpoints.
In addition, their computational cost at each break point is much higher
than a matrix-vector multiplication.

\section{Conclusion and discussions}
\label{sec:discussions}

This paper studied a proximal-gradient homotopy method for solving the
$\ell_1$-regularized least squares problems, focusing on its important 
application in sparse recovery. 
For such applications, the objective function is not strongly convex; 
hence the standard single-stage proximal gradient methods can only obtain 
relatively slow convergence rate. 
However, we have shown that under suitable conditions for sparse recovery, 
all iterates of the proximal-gradient homotopy method 
along the solution path are sparse.
With this extra sparsity structure, the objective function becomes effectively
strongly convex along the solution path, and thus a geometric rate of 
convergence can be achieved using the homotopy approach. 
Our theoretical analysis are supported by several numerical experiments.

We commented in the numerical experiments that accelerated gradient methods
cannot automatically exploit restricted strong convexity.
As discussed in \cite[Section~2.2]{Nesterov04book} and 
\cite{Nesterov07composite}, they need to explicitly use the strong convexity 
parameter, or a non-trivial lower bound of it,
to obtain geometric convergence.
In order to exploit restricted strong convexity in the $\ell_1$-LS problem 
with $m<n$, accelerated gradient methods need an extra facility
to come up with an explicit estimate of the restricted convexity parameter 
on the fly.
Nesterov gave some suggestions along this direction in 
\cite{Nesterov07composite}, and strategies such as periodic restart
have been studied recently \cite{GuLimWu09,BeckerCandesGrant11}.
However, an in-depth investigation on this matter is beyond the scope of 
this paper.

\bibliographystyle{alpha}
\bibliography{homotopy}

\end{document}